\documentclass[12pt,reqno]{amsart}

\usepackage[utf8]{inputenc}
\usepackage[T1]{fontenc}
\usepackage{enumerate}
\usepackage{hyperref}
\hypersetup{
  colorlinks   = true,
  citecolor    = blue,
  linkcolor    = blue
}

\usepackage{amsmath,amsfonts,amsthm,amssymb,color,tikz, comment}
\usepackage{mathrsfs}   

\usepackage{pdfsync}
\usepackage[font={scriptsize}]{caption}

\usepackage[left=1in, right=1in, top=1.1in,bottom=1.1in]{geometry}
\newcommand{\der}{\delta}

\newcommand{\hpsi}{\hat{\psi}}

\newcommand{\E}{\mathbb E}
\newcommand{\R}{\mathbb R}
\newcommand{\N}{\mathbb N}
\newcommand{\bP}{\mathbb P}

\newcommand{\Z}{\mathbb Z}
  
\newcommand{\be}{\mathbf{E}}

\newcommand{\bp}{\mathbf{P}}

\newcommand{\cb}{\mathcal B}
\newcommand{\cac}{\mathcal C}

\newcommand{\cd}{\mathcal D}
\newcommand{\ce}{\mathcal E}
\newcommand{\cf}{\mathcal F}
\newcommand{\cg}{\mathcal G}
\newcommand{\ch}{\mathcal H}

\newcommand{\ck}{\mathcal K}

\newcommand{\cm}{\mathcal M}
\newcommand{\cn}{\mathcal N}

\newcommand{\cs}{\mathcal S}

\newcommand{\al}{\alpha}

\newcommand{\ep}{\varepsilon}

\newcommand{\ga}{\gamma}

\newcommand{\ka}{\kappa}
\newcommand{\la}{\lambda}

\newcommand{\si}{\sigma}

\newcommand{\lp}{\left(}
\newcommand{\rp}{\right)}
\newcommand{\lc}{\left[}
\newcommand{\rc}{\right]}
\newcommand{\lcl}{\left\{}
\newcommand{\rcl}{\right\}}
\newcommand{\lln}{\left|}
\newcommand{\rrn}{\right|}
\newcommand{\lla}{\left\langle}
\newcommand{\rra}{\right\rangle}

\newtheorem{theorem}{Theorem}[section]

\newtheorem{definition}[theorem]{Definition}

\newtheorem{hypothesis}[theorem]{Hypothesis}
\newtheorem{lemma}[theorem]{Lemma}
\newtheorem{notation}[theorem]{Notation}

\newtheorem{proposition}[theorem]{Proposition}

\theoremstyle{remark}
\newtheorem{remark}[theorem]{Remark}

\theoremstyle{remark}

\newcommand{\bean}{\begin{eqnarray*}}
\newcommand{\eean}{\end{eqnarray*}}
\newcommand{\ben}{\begin{enumerate}}
\newcommand{\een}{\end{enumerate}}
\newcommand{\beq}{\begin{equation}}
\newcommand{\eeq}{\end{equation}}

\begin{document}

\title[SHE driven by rough spatial noise]{Quenched asymptotics for a 1-d stochastic heat equation driven by a rough spatial noise}

\address{Prakash Chakraborty: Department of Statistics,
  Purdue University,
  150 N. University Street,
  W. Lafayette, IN 47907,
  USA.}
\email{chakra15@purdue.edu}

\address{Xia Chen:  Department of Mathematics, University of Tennessee Knoxville,  TN 37996-1300, USA}
\email{xchen@math.utk.edu}
\thanks{X. Chen is partially supported by a grant 
from the Simons Foundation \#244767}

\address{Bo Gao:  Department of Mathematics, University of Tennessee Knoxville,  TN 37996-1300, USA}
\email{bgao3@vols.utk.edu}

\address{Samy Tindel: Department of Mathematics,
  Purdue University,
  150 N. University Street,
  W. Lafayette, IN 47907,
  USA.}
\email{stindel@purdue.edu}

\thanks{S. Tindel is supported by the NSF grant  DMS-1613163}

%\date{\today}

\author[P. Chakraborty, X. Chen, B. Gao, S. Tindel]
{Prakash Chakraborty, Xia Chen, Bo Gao, Samy Tindel}

\begin{abstract}
In this note we consider the parabolic Anderson model in one dimension with time-independent fractional noise $\dot{W}$ in space. We consider the case $H<\frac{1}{2}$ and get existence and uniqueness of solution. In order to find the quenched asymptotics for the solution we consider its Feynman-Kac representation and explore the asymptotics of the principal eigenvalue for a random operator of the form $\frac{1}{2} \Delta + \dot{W}$.
\end{abstract}

\maketitle

{
\hypersetup{linkcolor=black}
 \tableofcontents 
}

\section{Introduction}
The non trivial effects of random perturbations on the spectrum of the Laplace operator has been a fascinating object of research in the recent past. While a direct spectral analysis of perturbed Laplacians is possible in simple and regular enough cases \cite{GK,Ki}, the problem is often addressed through the  large time behavior of the so-called parabolic Anderson model. More specifically the parabolic Anderson model is a stochastic heat equation of the following form:
\beq\label{eq:she-1}
\dfrac{\partial u_{t}(x)}{\partial t} = \dfrac{1}{2} \Delta u_{t}(x) + u_{t}(x) \, \dot{W}(x),
\eeq
where the noise $\dot{W}$ is a stationary spatial random field. Because of the linear form of the noise term, it is possible under certain regularity conditions to express the solution of \eqref{eq:she-1} using a Feynman-Kac representation. Related to this representation, the asymptotic behavior of $u_{t}(x)$ as $t$ goes to $\infty$ gives some insight on the spectrum of the operator $\frac12 \Delta + \dot{W}$.

In the spatially discrete setting with a discrete Laplacian, asymptotic equivalents for the solution of equation~\eqref{eq:she-1} have been studied at length in \cite{mem_carmona_molchanov} and \cite{konig_book}. In particular, if $u_{t}(x)$ is the solution under the discrete setup in $\Z^{d}$ and $U(t) = \sum_{z \in \mathbb{Z}^d} u_{t}(x)$ is the total mass, then it has been proven that both $\frac{1}{t} \log (u_{t}(x))$ and $\frac{1}{t}\log (U(t))$ converge almost surely under certain regularity assumptions. Any information about those limits can then be translated into an information about the principal eigenvalue of $\frac12 \Delta + \dot{W}$.

In the spatially continuous setting, the picture is not as clear. Indeed, the large time behavior of the solution $u$ to equation~\eqref{eq:she-1} has been analyzed in \cite{carmona_molchanov} and \cite{gartner_konig_molchanov}. In particular, when the noise is Gaussian with a smooth covariance structure given by $\ga(x) = \textnormal{Cov} ( \dot{W}(0) \dot{W}(x))$ satisfying
$\lim_{|x| \to \infty} \ga(x) = 0$, then we have for $x \in \R^d$
\beq\label{eq:lyapou-smooth-noise}
\lim_{t \to \infty} \dfrac{1}{t \sqrt{\log t}} \log u_{t}(x) = \sqrt{2 d \ga(0)}\hspace{0.2in} \textnormal{a.s.}
\eeq
The fact that the renormalization in \eqref{eq:lyapou-smooth-noise} is of the form $t \sqrt{\log t}$ suggests that the principal eigenvalue of $\frac12 \Delta + \dot{W}$ is divergent, which is confirmed in \cite{AC,Mk} by asymptotics on large boxes  performed for the white noise.

Motivated by the examples above, non-smooth cases of equation~\eqref{eq:she-1} under the setting of generalized Gaussian fields have been analyzed in \cite{chen-ann14}. Namely, the reference \cite{chen-ann14} handles the case of a centered Gaussian noise $W$ whose covariance function $\Lambda$ is defined informally (see Section~\ref{sec:noise} for more precise definition) by
\begin{equation}\label{eq:Intronoise_cov}
\mathbf{E} \left[ W(\phi) W(\psi) \right] = \int_{\R^2} \phi(x) \psi(y) \Lambda(x-y) dx dy ,
\end{equation} 
for all infinite differentiable functions $\phi$ with compact support. The class of functions $\Lambda$ considered in \cite{chen-ann14} are continuous on $\R^d\setminus {\{0\}}$, bounded away from $0$ with a singularity at $0$ measured by $\Lambda(x) \sim c(\Lambda){|x|}^{-\al}$ with $\al\in(0,1)$ as $x \downarrow 0$. In this framework, the following almost sure renormalization result is proved in~\cite{chen-ann14} for any $x \in \R^d$:
\beq\label{eq:intro-1}
\lim_{t \to \infty} \frac{1}{t(\log t)^{\frac{2}{4-\al}}} \log u_{t}(x) = c_{\al}
\hspace{0.2in}\textnormal{a.s.} \ ,
\eeq
with an explicit constant $c_{\al}$.
Notice that this result is also applicable under a fractional white noise with Hurst parameter $H > \frac{1}{2}$. Namely, considering $d=1$ for simplicity, relation~\eqref{eq:intro-1} holds for a fractional Brownian noise $W$ with $\al = 2-2H$ (that is a renormalization of the form $t(\log t)^{\frac{1}{1+H}}$). 

In this note we aim to carry forward the asymptotic result~\eqref{eq:intro-1} to very singular environments. Specifically,  we consider a fractional noise $W$ as in \cite{chen-ann14}, but we allow the Hurst parameter to be less than $\frac{1}{2}$ (so that our noise is rougher than white noise). Going back to expression~\eqref{eq:Intronoise_cov}, we assume that $\Lambda$ is a positive definite distribution whose Fourier transform $\mathcal{F} \Lambda = \mu$ is a tempered measure given by $\mu(d \xi) = C_H{|\xi|}^{1-2H}d\xi$. That is for test functions $\phi$ and $\psi$ we have
\beq\label{eq:Intronoise_cov_fourier}
\be \lc W(\phi) W(\psi) \rc = \int_{\R} \cf \phi(\xi) \overline{\cf \psi(\xi)} \mu(d \xi).
\eeq

Let us first notice that equation~\eqref{eq:she-1} driven by a fBm with $H < \frac{1}{2}$ is not explicitly solved in the literature. As we will see, one can give a pathwise meaning, in a Young type sense, to the solution of equation~\eqref{eq:she-1}. Namely we show that $t \mapsto u_t$ can be seen as a continuous function with values in a weighted Besov space (we refer to \cite{mourrat-weber} for a complete definition of weighted Besov spaces). We will set up a fixed point argument in those weighted spaces and obtain the following result (see Theorem~\ref{thm:exi+uniq} for a more precise formulation).
\begin{theorem}\label{thm:Introthm-1}
Let $W$ be the Gaussian noise considered in \eqref{eq:Intronoise_cov_fourier} with $H \in \lp 0, \frac{1}{2}\rp$. Let $u_0$ be an initial condition lying in a weighted Besov H\"older space (see Definition~\ref{def:wt-besov} or a more detailed description). Then there exists a unique solution to \eqref{eq:she-1} in a space of continuous functions with values in Besov spaces, and where the integral with respect to $W$ is understood in the Young sense.
\end{theorem} 
 
Once we have solved \eqref{eq:she-1} , we will give a property of the (formal) operator $\frac{1}{2} \Delta + \dot{W}$ which is reminiscent of the density of states results contained e.g. in \cite{GK, Ki}. 
%Next we study the principal eigenvalue corresponding to the random operator $\frac{1}{2}\Delta + \dot{W}$. 
The result we obtain can be summarized informally in the following theorem:
\begin{theorem}\label{thm:Introthm-2}
Let $\la_{\dot{W}}(Q_t)$ be the principal eigenvalue of the random operator $\frac{1}{2}\Delta + \dot{W}$ over a restricted space of functions having compact support on $(-t,t)$. Then the following limit holds:
\beq\label{eq:Introthm-2}
\lim_{t \to \infty} \dfrac{\la_{\dot{W}} (Q_t)}{(\log t)^{\frac{1}{1+H}}} = {\lp 2 c_H \ce \rp}^{\frac{1}{1+H}}  \hspace{0.2in} \textnormal{a.s,} 
\eeq
with a strictly positive constant $\ce$ defined by
\beq\label{eq:ce-0}
\ce = \sup_{g \in \cg(\R)} \int_{\R} {\lln\int_{\R} e^{\imath \la x}  g^2(x) dx \rrn}^2 {\lln \la \rrn}^{1-2H} d\la 
\eeq
where $\cg (\R)$ is the space of all Schwartz functions satisfying ${\|g\|}_2^2 + \frac{1}{2} {\|g'\|}_2^2 = 1$.
\end{theorem}

Using a Feynman-Kac representation for the solution $u$ of \eqref{eq:she-1}, our next step will be to 
relate the logarithmic behavior of $u_{t}$ to the principal eigenvalue $\la_{\dot{W}} (Q_t)$.
This is the contents of the following theorem:
\begin{theorem}\label{thm:Introthm-3}
Let $W$ be the Gaussian noise defined by \eqref{eq:Intronoise_cov_fourier} for $H < \frac{1}{2}$, and consider the unique solution $u$ to \eqref{eq:she-1}. Then for all $x \in \R$ we have
$$
\lim_{t \to \infty} \dfrac{\frac{1}{t} \log(u_t(x))}{\la_{\dot{W}} (Q_t)} = 1, \hspace{0.2in}\textnormal{a.s.}
$$
\end{theorem}

As the reader might conceive, our main asymptotic result will be a simple consequence of Theorems~\ref{thm:Introthm-2} and \ref{thm:Introthm-3}. It gives a generalization of \eqref{eq:intro-1} to the case $H < \frac{1}{2}$. 

\begin{theorem}\label{thm:Introthm-4}
Under the same conditions as in Theorem~\ref{thm:Introthm-3} and for $H < \frac{1}{2}$ we have
$$
\lim_{t \to \infty} \dfrac{\log(u_t(x))}{t (\log t)^{\frac{1}{1+H}}} = (2 c_H \ce)^{\frac{1}{1+H}}, \hspace{0.2in}\textnormal{a.s.}
$$
\end{theorem}

Let us say a few words about the methodology we have resorted to in order to get our main results.

\noindent
\emph{(i)}
Theorem~\ref{thm:Introthm-2} is obtained by splitting the eigenvalue problem into small intervals, similarly to what is performed in other parabolic Anderson model studies (see e.g \cite{konig_book} and \cite{chen-ann14}). Then on each subdomain we combine some variational arguments with supremum computations for Gaussian processes. An extra care is 
required in our case, due to the singularity of our noise.

\noindent
\emph{(ii)}
Theorem~\ref{thm:Introthm-3} relies on a Feynman-Kac representation of $u_t(x)$, whose main ingredient is an integrability property established thanks to a subtle sub-additive argument (see Proposition~\ref{prop:exp_moment_conv} below). Once this Feynman-Kac representation (involving a Brownian motion $B$) is given, a probabilistic cutoff procedure on the underlying Brownian motion $B$ allows to reduce the logarithmic behavior of $u_{t}(0)$ to the quantity $\la_{\dot{W}}(Q_t)$.

\noindent
\emph{(iii)}
As mentioned above, Theorem~\ref{thm:Introthm-4} is an easy consequence of Theorems~\ref{thm:Introthm-2} and~\ref{thm:Introthm-3}.

\smallskip

Eventually let us highlight the fact that Theorem~\ref{thm:Introthm-4} provides a rather complete description of the asymptotic behavior of $\log(u_t(x))$ in dimension 1. A very challenging situation would be to handle the case of a rough noise in dimension 2 or higher. In this case it is a well known fact that a renormalization procedure is needed to define the solution $u$ of \eqref{eq:she-1}, as shown e.g. in \cite{De}. The effect of this kind of renormalization procedure on the principal eigenvalue of $\frac{1}{2}\Delta + \dot{W}$ has been partially investigated for the space white noise when $d=2$ in \cite{AC}.

\noindent
\textbf{Notations.} We denote by $p_t(x)$ the one-dimensional heat kernel $p_t (x) = \left(2 \pi t\right)^{-1/2}e^{-|x|^2/{2t}}$, for any $t > 0$, $x \in \R$. The space of real valued infinitely differentiable functions with compact support is denoted by $\mathcal{D}(\R)$. The space of Schwartz functions is denoted by $\mathscr{S}(\R)$. Its dual, the space of tempered distributions, is $\mathscr{S}'(\R)$. The Fourier transform is defined as
\begin{equation*}
\mathcal{F}u(\xi) = \hat{u}(\xi) = \int_{\R} e^{-\iota \langle x, \xi \rangle} u(x) dx.
\end{equation*}
The inverse Fourier transform is ${\mathcal{F}}^{-1}u (\xi) = (2\pi)^{-1} \mathcal{F}(-\xi)$. 
Denote by $l$ the following probability density function in $\mathcal{S}(\mathbb{R})$:
$$
l(x) = c \exp \lp -\dfrac{1}{1-x^2} \rp \mathbf{1}_{(|x|<1)},
$$
where $c$ is a normalizing constant such that $\int_{\mathbb
R}l(x)dx=1$. For every $\ep > 0$, let the set of mollifiers generated by $l$ be given by $l_\varepsilon(x)=\varepsilon^{-1}l(\varepsilon^{-1}x)$.
Observe that, owing to the fact that $l$ is a probability measure, we have $\lim_{\xi\to 0}\cf l(\xi)=1$ and $\cf l(\xi)\le1$ for all $\xi\in\R$.

\section{Preliminaries}
This section is devoted to introduce the basic Besov spaces notions which will be used in the remainder of the paper. Observe that since we are dealing with a variable $x$ in the whole space $\R$, we will need to deal with weighted Besov spaces. The definitions and main properties of those spaces are borrowed from \cite{mourrat-weber}. 
\subsection{Besov spaces}
In this subsection we define some classes of weights which are compatible with our heat equation~\eqref{eq:she-1}. Two scales of weights will be used: stretched exponential weights and polynomial weights.
\begin{definition}\label{def:weights}
Let ${|x|}_{\ast} = \sqrt{1+|x|^2}$. Fix $\der \in (0,1)$ and $\theta \in (1,1/\der)$. Denote by $\mathcal{W}$ the class of weights consisting of:
\begin{itemize}
\item[(i)] the weights $w_{\ga}$ of the form $w_{\ga}(x) = e^{-\ga {|x|}_{\ast}^{\der}}$, with $\ga > 0$.
\item[(ii)] the weights $\hat{w}_{\si}$ of the form $\hat{w}_{\si}(x) = {|x|}_{\ast}^{-\si}$, with $\si > 0$.
\end{itemize}
\end{definition}

The definition of our Besov spaces depends heavily on a dyadic partition of unity. In order to handle weights as in Definition~\ref{def:weights} we have to work (as done in \cite{mourrat-weber}) with functions in the so-called Gevrey class, that we now proceed to define.
\begin{definition}\label{def:gevrey}
Let $\theta \geq 1$. We call $\mathcal{G}^{\theta}$, the set of infinitely differentiable functions $f:\R \to \R$ satisfying
\begin{center}
 for every compact $K$, there exists $C < \infty$ such that for every $n \in \mathbb{N}$,
 $$ \sup_K|\partial^n f| \leq C^{n+1} (n!)^{\theta}.$$
\end{center} 
We let $\mathcal{G}_c^{\theta}$ be the set of compactly supported functions in $\mathcal{G}^{\theta}$. 
\end{definition}

We are now ready to state the existence of a partition of unity in the Gevrey class $\cg_c^{\theta}$.
\begin{proposition}\label{prop:unity_partition}
One can construct two functions $\tilde{\chi},\chi \in \mathcal{G}_c^{\theta}$, taking values in $\left[0,1\right]$ and such that 
\begin{itemize}
\item[(i)] \textnormal{Supp} $\tilde{\chi} \subseteq \left[0,\frac{4}{3}\right]$ and \textnormal{Supp} $\chi \subseteq \left[\frac{3}{4}, \frac{8}{3}\right]$.
\item[(ii)] For all $\xi \in \R$, we have $\tilde{\chi}(\xi) + \sum_{k=0}^{\infty} \chi(2^{-k}{\xi}) = 1$. 
\end{itemize}
In the sequel we also set $\chi_k(\xi) = \chi(2^{-k}\xi)$ for $k \geq 0$.
\end{proposition}
With the partition of unity in hand, the blocks $\Delta_ku$ of the Besov type analysis can be defined as follows.

\begin{definition}\label{def:wt_blocks}
Set $\chi_{-1} = \tilde{\chi}$, and define for $k \geq -1$ and $u \in \cs(\R)$, 
\begin{equation*}
\Delta_k u = {\mathcal{F}}^{-1}(\chi_k \hat{u}).
\end{equation*}
\end{definition} 

Our analysis will rely on Besov spaces defined through the weighted blocks introduced in Definition~\ref{def:wt_blocks}. 
\begin{definition}\label{def:wt-besov}
Let $\chi$ and $\tilde{\chi}$ be the functions introduced in Proposition~\ref{prop:unity_partition}. For any $\kappa \in \R$, $w \in \mathcal{W}$, $p,q \in [1, \infty]$ and $f \in \cs(\R)$, we define weighted norms of $f$ in the following way:
\begin{equation}\label{eq:wt_Besov_norm}
{\|f\|}_{\mathcal{B}_{p,q}^{\kappa, w}} := \left[ \sum_{j=-1}^{\infty} \left( 2^{\ka j}{\|\Delta_j f\|}_{L^p_{w}} \right)^q \right]^{\frac{1}{q}},
\end{equation}
where $L^p_w$ is the weighted space $L^p(\R, w(x)dx)$. Denote the weighted Besov space $\mathcal{B}_{p,q}^{\ka, w}$ as 
$$
\mathcal{B}_{p,q}^{\ka, w} = \lcl f \in \cs(\R); {\|f\|}_{\mathcal{B}_{p,q}^{\ka, }} < \infty \rcl.
$$
\end{definition}

\begin{remark}
Notice that as in \cite{triebel3}, we define ${\|f\|}_{L^p(\R^d; w(x)dx)}$ as ${\|fw\|}_{L^p(\R^d)}$. This is slightly different from \cite{mourrat-weber}, but yields similar results.
\end{remark}

In the next section we will solve the heat equation in a weighted Besov space whose weight is varying with time. We now define this kind of space.
\begin{notation}\label{not:c_p}
Let $\la$ and $\si$ be two strictly positive constants. For $t \geq 0$ we define $v$ as the function $v_t = w_{\la+ \si t}$, where we recall that $w_{\ga}$ is introduced in Definition~\ref{def:weights}. We consider an additional parameter $\ka_u > 0$ and $p \in [1,\infty)$. Then the space $\cac_{p}^{\ka_u, \la, \si}$ is defined by 
\begin{equation*}
\mathcal{C}_{q}^{\ka_u, \la, \si} = \lcl f \in \mathcal{C}([0,T] \times \R); {\|f_t\|}_{{\mathcal{B}}_{q,\infty}^{\ka_u, v_t}} \leq c_f \rcl.
\end{equation*}
\end{notation}

\subsection{Description of the noise}\label{sec:noise}
The noise driving equation~\eqref{eq:she-1} is considered as a centered Gaussian family $\left\lbrace W(\phi), \phi \in \mathcal{D}(\R) \right\rbrace$ on a complete probability space $(\Omega, \cf, \bp)$ with the following covariance structure:
\begin{equation}\label{eq:b}
\mathbf{E} \left[ W(\phi) W(\psi) \right] = \int_{\R^2} \phi(x) \psi(y) \Lambda(x-y) dx dy,
\end{equation} 
where $\Lambda:\R \mapsto \R_{+}$ is a non-negative definite distribution. In fact the covariance structure of $W$ is better described in Fourier modes. Indeed, the distribution $\Lambda$ can be seen as the inverse Fourier transform of a measure $\mu$ on $\R$ defined by 
$$
\mu(d\xi) = c_H {|\xi|}^{1-2H} d \xi.
$$ 
Then for $\phi, \psi \in \cd(\R)$ we have 
\beq\label{eq:cov_noise_Fourier}
\be \lc W(\phi) W(\psi) \rc = \int_{\R} \cf \phi(\xi) \overline{\cf \psi(\xi)} \mu(d \xi).
\eeq
It can be shown that \eqref{eq:cov_noise_Fourier} defines an inner product on $\cd(\R)$. We call $\ch$ the completion of $\cd(\R)$ with this inner product. It also holds that the variance of our noise $W$ has an alternate direct-coordinate representation (see e.g. \cite[relation (2.8)]{1d-rough}) in addition to the one suggested by \eqref{eq:cov_noise_Fourier}. Namely for $\phi \in \ch$, we have
\beq\label{eq:cov_noise_homSob}
\be {\lc W(\phi) \rc}^2 = {c}_H\int_{\R} \int_{\R} \dfrac{{\lln \phi(x+y) - \phi(x) \rrn}^2}{{|y|}^{2-2H}} dx dy. 
\eeq
	
The mapping $\phi \mapsto W(\phi)$ defined in $\mathcal{D}(\R)$ extends to a linear isometry between $\mathcal{H}$ and the Gaussian space spanned by $W$. This isometry will be denoted by
\beq\label{eq:integral_wiener}
W(\phi) = \int_{\R} \phi(x)W(dx).
\eeq

\begin{remark}\label{rem:mu}
Notice that the measure $\mu(d \xi) = c_H {|\xi|}^{1-2H} d\xi$ satisfies the following condition 
\begin{equation}\label{eq:mu}
\int_{\R} \frac{\mu(d\xi)}{1+|\xi|^{2(1-\alpha)}} < \infty, \quad\text{ for }\alpha < H.
\end{equation}
This relation will be crucial in order to see that $W$ belongs to a weighted Besov space in the proposition below.
\end{remark}

\begin{proposition}
For all $\kappa \in \left( 1-H, 1 \right)$ and every arbitrary $\si > 0$, W has a version in $\mathcal{B}_{\infty, \infty}^{-\kappa, \hat{w}_{\si}}$, where $\hat{w}_{\sigma}$ is given in Definition~\ref{def:weights} and $\mathcal{B}_{\infty, \infty}^{-\ka, \hat{w}_{\si}}$ is introduced in Definition~\ref{def:wt-besov}. In addition the random variable ${\|W\|}_{\mathcal{B}_{\infty, \infty}^{-\ka, \hat{w}_{\si}}}$ has moments of all orders. 
\end{proposition}

\begin{proof}
Consider $\ka$, $\ka'$ such that $\ka > \ka' > 1-\al$, where $\al$ is defined by \eqref{eq:mu}. Due to  \eqref{eq:mu} observe that this implies that we can consider any $\ka > 1-H$. For $q \geq 1$, denote the Besov space $\mathcal{B}_{2q, 2q}^{-\ka', \hat{w}_{\si}}$ by $\mathcal{A}_q$. 
The proof relies on the fact that for large enough $q$, $\mathcal{A}_q$ is continuously embedded in $\mathcal{B}_{\infty, \infty}^{-\ka, \hat{w}_{\si}}$, i.e.,
\beq\label{eq:embedding_eq}
{\|\dot{W}\|}_{\cb_{\infty, \infty}^{-\ka, \hat{w}_{\si}}} \lesssim {\|\dot{W}\|}_{\mathcal{A}_q}.
\eeq
Hence it is enough to work with ${\|\dot{W}\|}_{\mathcal{A}_q}$. 

Let us now evaluate the quantity ${\|\dot{W}\|}_{\mathcal{A}_q}$. To this aim, notice that ${\Delta}_j f(x) = \left[ K_j \ast f \right](x)$ where $K_j(z) = 2^{jd} {\mathcal{F}}^{-1} \chi(2^j z)$. Therefore, using the notation $K_{j,x}(y)=K_j(x-y)$ we obtain:
\beq\label{eq:noise_diffe_sum_besov-1}
\mathbf{E}\left[ {\|\dot{W}\|}_{\mathcal{A}_q}^{2q} \right] = \sum_{j \geq -1} 2^{-2qj\kappa '} \int_{\R} \mathbf{E} \left[ {|W(K_{j,x})|^{2q}} \right] \hat{w}_{\sigma}^{2q}(x)dx.
\eeq
Using the fact that $W(K_{j,x})$ is Gaussian we thus have
$$
\mathbf{E}\left[ {|W(K_{j,x})|}^{2q} \right] \leq c_q \mathbf{E}^q \left[ {|W(K_{j,x})|}^2 \right].
$$ 
Consequently, \eqref{eq:noise_diffe_sum_besov-1} can be recast as:
\begin{equation}\label{eq:c}
\mathbf{E}\left[ {\|\dot{W}\|}_{\mathcal{A}_q}^{2q} \right] \leq c_q \sum_{j \geq -1} 2^{-2qj\kappa'} \int_{\R} \mathbf{E}^q \left[ {|W(K_{j,x})|}^2 \right] \hat{w}_{\si}^{2q}(x)dx.
\end{equation}
Now let us work with $\mathbf{E}\left[ {|W(K_{j,x})|}^2 \right]$. According to \eqref{eq:b} we have
$$
\mathbf{E}\left[ {|W(K_{j,x})|}^2 \right] =\int_{\R} |\mathcal{F}K_{j,x}(\xi)|^2 \mu(d\xi).
$$
Let us introduce a new measure $\nu$ on $\R$ defined by $\nu(d\xi) = \frac{\mu(d\xi)}{1+{|\xi|}^{2(1-\al)}}$. Notice that due to~\eqref{eq:mu}, $\nu$ is a finite measure.
Since $K_j = \mathcal{F}^{-1} \chi_j$ and the support of $\chi$ is in a closed interval, say $[a,b]$, we obtain:
\begin{align}\label{eq:d}
\mathbf{E}\left[ {|W(K_{j,x})|}^2 \right] &= \int_{\R}{\left\vert\chi(2^{-j} \xi)\right\vert}^2 \mu(d\xi) \leq \int_{\R} \mathbf{1}_{[0, 2^j b]}(|\xi|) \left( 1+{|\xi|}^{2(1-\al)} \right) \nu(d \xi)\\
&\leq \nu \left( \left[ 0, 2^j b \right]\right) (1+(2^j b)^{2(1-\al)}) \leq c_{\mu} 2^{2(1-\al)j}.
\end{align}

\noindent
Therefore plugging \eqref{eq:d} into \eqref{eq:c} and recalling that $1-\al < \ka' < \ka$, we get:
\begin{align}\label{eq:e}
\mathbf{E}\left[ {\|\dot{W}\|}_{\mathcal{A}_q}^{2q} \right] &\leq c_q \sum_{j \geq -1} 2^{-2qj\kappa'} \int_{\R} c_{\mu}^q 2^{(1-\al)jq} \hat{w}_{\si}^{2q} (x) dx \nonumber \\
&=C_{q,\mu} \left(\int_{\R} \hat{w}_{\si}^{2q}(x)dx\right) \sum_{j \geq -1} 2^{2qj(1-\al-\ka')}. 
\end{align}
Owing to the Definition~\ref{def:weights} of $\hat{w}_{\si}$, it is now readily checked that the right hand side of \eqref{eq:e} is convergent whenever $q$ is large enough.

Similar calculations as the ones leading to \eqref{eq:e} also show that the random variable ${\|W\|}_{\cb_{\infty, \infty}^{-\ka, \hat{w}_{\si}}}$ has moments of all orders. 
\end{proof}

\section{Pathwise solution}
Now that we have proved that our noise $\dot{W}$ is almost surely an element of $\cb_{\infty, \infty}^{-\ka, \hat{w}_{\si}}$, we will transform our stochastic eq.~\eqref{eq:she-1} into a deterministic one, which will be solved in the Riemann-Stieltjes sense. We first label an assumption on a general distribution driving the heat equation.
\begin{hypothesis}\label{hyp:pathwise_W}
Let $\delta \in (0,1)$ be a fixed constant and $\si > 0$ be an arbitrarily small constant. We consider a distribution $\mathscr{W}$ on $\R$ such that $\mathscr{W} \in {\mathcal{B}}^{-\ka, \hat{w}_{\si \der}}_{\infty, \infty}$ with $\ka \in (0,1)$.
\end{hypothesis}
\begin{remark}
The constant $\delta \in (0,1)$ in Hypothesis~\ref{hyp:pathwise_W} is related to the exponential weights in Definition~\ref{def:weights}.
\end{remark}
We now introduce the notion of solution for equation~\eqref{eq:she-1} which will be considered in the sequel.
\begin{definition}
Let $\mathscr{W}$ be a distribution satisfying Hypothesis~\ref{hyp:pathwise_W}. Let $u \in \cac_{q}^{\ka_u, \la, \si}$ for $\la, \si > 0$ and $\ka_u \in (\ka, 1)$, where $\cac_{q}^{\ka_u, \la, \si}$ is introduced in Notation~\ref{not:c_p}. Consider an initial condition $u_0 \in \mathcal{B}^{\ka_u, v_0}_{q, p}$ where we recall $v_t = w_{\la + \si t}$. We say that $u$ is a mild solution to equation 
\begin{equation}\label{eq:she}
\dfrac{\partial u}{\partial t} = \frac{1}{2} \Delta u + u \mathscr{W}
\end{equation}
with initial condition $u_0$, if it satisfies the following integral equation 
\begin{equation}\label{eq:integral-form}
u_t = p_t u_0 + \int_0^t p_{t-s}(u_s \mathscr{W}) ds.
\end{equation}
\end{definition}
\begin{remark}
In \eqref{eq:integral-form}, we implicitly assume that the product of distributions $u \cdot \mathscr{W}$ is well defined. This will be treated in the forthcoming Lemma~\ref{lem:fact4}. 
\end{remark}
Before we can solve equation~\eqref{eq:integral-form}, we list a few results which would prove useful later. The first one recalls the action of the heat semigroup on weighted Besov spaces.
\begin{lemma}\label{lem:fact1}
The following smoothing effect of the heat flow is valid in Besov spaces:
Let $\hat{\ka} \geq \ka$ be real numbers, $\ga_0 > 0$ and $q \in [1, \infty]$. Then there exists $C < \infty$ such that uniformly over $\ga \leq \ga_0$ and $t > 0$,
\begin{equation*}
{\|p_t f\|}_{\mathcal{B}_{q,\infty}^{\hat{\ka}, w_{\ga}}} \leq C t^{- \frac{\hat{\ka}-\ka}{2}} {\|f\|}_{\mathcal{B}_{q,\infty}^{\ka, w_{\ga}}}
\end{equation*}
\end{lemma}
\begin{proof}
See \cite[Proposition 3.11]{mourrat-weber}.
\end{proof}
We now give a result on comparison of Besov norms for different weights $w$.
\begin{lemma}\label{lem:fact2}
Let $w_1, w_2 \in \mathcal{W}$ be such that $w_1 \leq w_2$. Then for every $f \in \mathcal{B}_{p,q}^{\ka, w_2}$ we have
\begin{equation*}
{\|f\|}_{\mathcal{B}_{p,q}^{\ka, w_1}} \leq {\|f\|}_{\mathcal{B}_{p,q}^{\ka, w_2}}
\end{equation*}
\end{lemma}
\begin{proof}
Follows easily from Definition~\ref{def:wt-besov}.
\end{proof}
Our next preliminary lemma is an elementary comparison between the weights corresponding to Definition~\ref{def:weights}.
\begin{lemma}\label{lem:fact3}
Recall that the weight $v_t = w_{\la + \si t}$ has been defined for $t \geq 0$ in Notation~\ref{not:c_p}. Then for $0 \leq s < t$ we have
\begin{equation*}
v_t \leq c_{\si} {|t-s|}^{-\si} v_s \hat{w}_{\der \si}
\end{equation*}
\end{lemma}
\begin{proof}
For $0\leq s < t$, observe that $v_t = v_s e^{-\si (t-s) {|x|}_{\ast}^{\der}}$. Then we use the fact that there exists a constant $c_{\al}$ such that 
\begin{equation*}
0 \leq x^{\al} e^{-sx} \leq \dfrac{c_{\al}}{{s}^{\al}}~\text{ for }~x,\al,s \in \R_{+}
\end{equation*}
Consequently $e^{-\si(t-s){|x|}_{\ast}^{\der}} \leq c_{\si} {|t-s|}^{-\si} {|x|}_{\ast}^{-\si \der}$ which implies $v_t \leq c_{\si} {|t-s|}^{-\si} v_s \hat{w}_{\si \der}$ 
\end{proof}
Let us recall the definition of products of distributions within the weighted Besov spaces framework.
\begin{lemma}\label{lem:fact4}
Let $\al < 0 < \beta$ be such that $\al + \beta > 0$. In addition, consider $p,q \in [1, \infty]$ and $\nu \in [0,1]$. Let $p_1, p_2 \in [1, \infty]$ be such that 
\begin{equation*}
\dfrac{1}{p_1} = \dfrac{\nu}{p} ~\text{ , }~ \dfrac{1}{p_2} = \dfrac{1-\nu}{p} \quad\text{ and }\quad w=w_{\ga} \hat{w}_{\si}.
\end{equation*}
Then the mapping $(f,g) \mapsto fg$ can be extended to a continuous linear map from $\mathcal{B}_{p_1,q}^{\al, w_\ga} \times \mathcal{B}_{p_2,q}^{\beta, \hat{w}_{\si}}$  to $\mathcal{B}_{p,q}^{\al, w}$. Moreover there exists a constant $C$ such that
\begin{equation*}
{\|fg\|}_{\mathcal{B}_{p,q}^{\al, w}} \leq C {\|f\|}_{\mathcal{B}_{p_1, q}^{\al, w_{\ga}}} {\|g\|}_{\mathcal{B}_{p_2, q}^{\beta, \hat{w}_{\si}}}.
\end{equation*}
\end{lemma}
\begin{proof}
The proof is similar to that of \cite[Corollary 3.21]{mourrat-weber}.
\end{proof}
We also include the following extension of Gronwall's Lemma taken from \cite[Lemma 15]{dalang} which will be required in order to show existence of our solution.
\begin{lemma}\label{lem:gronwall}
Let $g:[0,T] \mapsto \R_{+}$ be a non-negative function such that $\int_0^T g(s)ds < \infty$. Let $\lp f_n, n\in \N \rp$ be a sequence of non-negative functions on $[0,T]$ and $k_1, k_2$ be non-negative numbers such that for $0 \leq t \leq T$,
\begin{equation}
f_n(t) \leq k_1 + \int_0^t (k_2 + f_{n-1}(s))g(t-s) ds.
\end{equation}
If $\sup_{0 \leq s \leq T} f_0(s) < \infty$, then $\sup_{n \geq 0} \sup_{0 \leq t \leq T} f_n(t) < \infty$, and if $k_1=k_2=0$, then $\sum_{n \geq 0} f_n(t)$ converges uniformly on $[0,T]$.
\end{lemma}

We are ready to state our main result about existence and uniqueness of solution for our abstract heat equation~\eqref{eq:she}. 
\begin{proposition}\label{prop:exi+uni}
Let $\mathscr{W}$ be a distribution as in Hypothesis \ref{hyp:pathwise_W}. Consider $\la > 0$ and $q \geq 1$. Then there exists a unique solution to equation~\eqref{eq:integral-form} lying in $\mathcal{C}_{q}^{\ka_u, \la, \si}$ where $\ka_u \in (\ka,1)$ and where $\cac_p^{\ka_u, \la, \si}$ is defined in Notation~\ref{not:c_p}.
\end{proposition}

\begin{proof}
We will follow a standard Picard iteration scheme to prove our result. 
Consider a small time interval $\left[0,\tau\right]$ where $\tau$ is to be fixed later. We restrict all spaces and corresponding norms on this time interval. 
Define $u^{(0)} \equiv u_0$ and for $n \geq 0$ set 
\begin{equation}\label{eq:iter_eq}
u_t^{(n+1)} = \int_0^t p_{t-s} (u_s^{(n)}\mathscr{W})ds.
\end{equation}
Fix $\ka_u \in (\ka,1)$ and consider $\der u_t^{(n)} = u_t^{(n+1)} - u_t^{(n)}$. Observe that from \eqref{eq:iter_eq} and then applying Lemma~\ref{lem:fact1} we obtain
\begin{align*}
{\|\der u^{(n+1)}_t\|}_{{\mathcal{B}}^{\ka_u, v_t}_{q,\infty}} &\leq \int_0^t {\|p_{t-s}(\der u^{(n)}_{s}\mathscr{W})\|}_{\mathcal{B}^{\ka_u, v_t}_{q,\infty}}\\
& \leq C \int_0^t {|t-s|}^{-\frac{\ka_u + \ka}{2}} {\|\der u^{(n)}_{s} \mathscr{W}\|}_{\mathcal{B}^{-\ka, v_t}_{q,\infty}} ds.
\end{align*}
where here and in the following $C$ is a generic constant which may change in subsequent steps. Now applying Lemmas~\ref{lem:fact2} and \ref{lem:fact3} we get
$$
{\|\der u_t^{(n+1)}\|}_{\mathcal{B}_{q,\infty}^{\ka_u, v_t}} \leq C \int_0^t (t-s)^{-\frac{\ka_u + \ka}{2} - \si} {\|\der u_s^{(n)} \mathscr{W}\|}_{\mathcal{B}_{q,\infty}^{-\ka, v_s \hat{w}_{\der \si}}} ds.
$$
Using $\nu = 1$ in Lemma~\ref{lem:fact4} and observing $\ka_u > \ka$, we find 
$$
{\|\der u^{(n)}_{s} \mathscr{W}\|}_{\mathcal{B}^{-\ka, v_s \hat{w}_{\delta \si}}_{q,\infty}} \leq {\|\der u^{(n)}_s\|}_{\mathcal{B}^{\ka_u, v_s}_{q,\infty}} {\|\mathscr{W}\|}_{\mathcal{B}^{-\ka, \hat{w}_{\delta \si}}_{\infty, \infty}}.
$$
Consequently, 
\begin{equation}\label{eq:alpha}
{\|\der u^{(n+1)}_t\|}_{\mathcal{B}^{\ka_u, v_t}_{q,\infty}} \leq C {\|\mathscr{W}\|}_{\mathcal{B}^{-\ka, \hat{w}_{\delta \si}}_{\infty, \infty}} \int_0^t \dfrac{{\|\der u^{(n)}_s\|}_{\mathcal{B}^{\ka_u, v_s}_{q,\infty}}}{{|t-s|}^{(\ka_u + \ka)/2 + \si}} ds.
\end{equation}
Observe that 
$$\sup_{0 \leq s \leq \tau} {\|\der u^{(0)}_s\|}_{\mathcal{B}^{\ka_u, v_s}_{q,\infty}} = \sup_{0 \leq s \leq \tau} {\|u^{(1)}_s -u^{(0)}_s\|}_{\mathcal{B}^{\ka_u, v_s}_{q,\infty}} = \sup_{0 \leq s \leq \tau} {\|p_{s} u_0 - u_0\|}_{\mathcal{B}^{\ka_u, v_s}_{q,\infty}}.
$$ 
Also recall that $v_s = w_{\la+\si s}$, where the weight $w_{\la + \si s}$ has been defined in Definition~\ref{def:weights} above. Consequently 
${\|p_s u_0\|}_{\mathcal{B}^{\ka_u, v_s}_{q,\infty}} \leq {\|p_s u_0\|}_{\mathcal{B}^{\ka_u, v_0}_{q,\infty}}$. Thus, owing to Lemma~\ref{lem:fact1} and \ref{lem:fact2}, we have
$$
\sup_{0 \leq s \leq \tau} {\|\der u_s^{(0)}\|}_{\mathcal{B}^{\ka_u, v_s}_{q,\infty}} \leq \sup_{0 \leq s \leq \tau} {\|p_s u_0 - u_0 \|}_{\mathcal{B}^{\ka_u, v_0}_{q,\infty}} \leq C {\|u_0\|}_{\mathcal{B}^{\ka_u, v_0}_{q,\infty}}
$$
which is finite by our assumption on the initial condition. We can thus apply Gronwall's Lemma as stated in Lemma~\ref{lem:gronwall} to equation~\eqref{eq:alpha}. 
As a consequence we find $\sum_{n \geq 0} {\|\der u_s^{(n)}\|}_{\mathcal{B}^{\ka_u, v_s}_{q,\infty}}$ converges uniformly on $[0, \tau]$ and thus $u^{(n)}$ converges uniformly in $\mathcal{C}_{q}^{\ka_u, \la, \si}$. This proves existence of a solution on $[0, \tau]$ (observe that we don't need $\tau$ to be small for this step). 

In order to prove uniqueness, we can resort to the same techniques. Consider two solutions $u^1$ and $u^2$ in $\mathcal{C}_{q}^{\ka_u, \la, \si}$ and set $u^{12}=u_1-u_2$. We have to show $u^{12}\equiv 0$. Since we have
$$ %\label{eq:u_diff}
u_t^{12} = \int_0^t p_{t-s}(u_s^{12} \mathscr{W})ds,
$$
we obtain similarly to \eqref{eq:alpha}
$$
{\|u^{12}_t\|}_{\mathcal{B}^{\ka_u, v_t}_{q,\infty}} \leq C {\|\mathscr{W}\|}_{\mathcal{B}^{-\ka, \hat{w}_{\delta \si}}_{\infty, \infty}} \int_0^t \dfrac{{\|u^{12}_s\|}_{\mathcal{B}^{\ka_u, v_s}_{q,\infty}}}{{|t-s|}^{(\ka_u + \ka)/2 + \si}} ds.
$$
Therefore, choosing $\si$ small enough we get:
$$
{\|u^{12}_{t}\|}_{\mathcal{B}^{\ka_u, v_t}_{q,\infty}} \leq \lp C {\|\mathscr{W}\|}_{\mathcal{B}^{-\ka, \hat{w}_{\delta \si}}_{\infty, \infty}} {\tau}^{\eta} \rp \sup_{0 \leq s \leq \tau} {\|u^{12}_s\|}_{\mathcal{B}^{\ka_u, v_s}_{q,\infty}}.
$$
where $\eta = 1-\lp \frac{\ka_u + \ka}{2} + \si \rp$. Then choosing $\tau$ small enough so that $( {\|\mathscr{W}\|}_{{\mathcal{B}}_{\infty, \infty}^{-\ka, \hat{w}_{\delta \si}}} {\tau}^{\eta} ) < 1$, we find $\|u^{12}_t\|_{\mathcal{B}^{\ka_u, v_t}_{q,\infty}} = 0$ for all $t \in [0, \tau]$. This achieves uniqueness on the small interval $[0, \tau)$.

In order to get global existence and uniqueness we observe that our considerations above do not depend on the initial condition of the solution. Hence one can repeat the proof on subsequent intervals of size $\tau$ to get the result. 
\end{proof}

We can now apply our general Proposition~\ref{prop:exi+uni} in order to solve our original equation~\eqref{eq:she-1}. 
\begin{theorem}\label{thm:exi+uniq}
Let $W$ be the centered Gaussian noise defined by \eqref{eq:cov_noise_Fourier}, with $H \in \lp 0, \frac{1}{2}\rp$ and consider $\ka \in (1-H,1)$. Let $u_0 \in \cb_{q,\infty}^{\ka_u, w_{\la}}$ for a given $\la > 0$ and $\ka_u \in (\ka, 1)$, where $w_{\la} = e^{-\la {|x|_{*}^{\delta}}}$ is defined in Definition~\ref{def:weights}. Consider the space $\cac_{q}^{\ka_u, \la, \si}$ introduced in Notation~\ref{not:c_p}. Then equation~\eqref{eq:she-1} admits a solution which is unique in $\mathcal{C}_q^{\ka_u, \la, \si}$.
\end{theorem} 

\section{Feynman-Kac representation}
In this section we shall establish a Feynman-Kac representation for the solution of \eqref{eq:she-1}, which will be at the heart of our Lyapounov computations. We first introduce some additional notations about random environments.

\begin{notation}\label{not:FK-1}
Let $B$ be a Brownian motion defined on a probability space $( \hat{\Omega}, \hat{\cf}, \bP )$, independent of the space $( \Omega, \cf, \bp )$ on which $W$ is defined. In the sequel we denote by $\be$ (resp. $\E$) the expectation on $( \Omega, \cf, \bp )$ (resp. $( \hat{\Omega}, \hat{\cf}, \bP )$). We will also write $\E_x$ when we want to highlight the initial value $x$ of the Brownian motion $B$.
\end{notation}

We now introduce the Feynman-Kac functional we shall use in order to represent the solution of \eqref{eq:she-1}. 
\begin{notation}
Let $W$ be the Gaussian noise defined by \eqref{eq:cov_noise_Fourier}. For $\ep > 0$ we set 
\beq\label{eq:V_t^ep}
V_{t}^{\ep}(x) = \int_0^t \int_{\R}l_{\ep} (B_r^x - y) W(dy) dr, 
\eeq
where $l_{\ep}$ stands for the $\ep$-mollifier generated from the standard bump function $l$ as given in the general notation of the Introduction.
We will also write, somehow informally
\beq\label{eq:V_t-def}
V_{t}(x) = \int_0^t W(\der_{B_s^x})ds = \int_0^t \int_{\R} \der_0 (B_r^x - y) W(dy) dr,
\eeq
which will be seen as a $L^2$-limit of the random variables $V_t^{\ep}$.
\end{notation}

We state the following lemma taken from \cite{Chen-Book} which will be used in the proof for Proposition~\ref{prop:exp_moment_conv}
\begin{lemma}\label{lem:subadditive-exp}
For any non-decreasing sub-additive process $Z_t$ defined on $( \hat{\Omega}, \hat{\cf}, \bP )$ with continuous path and with $Z_0 = 0$, the following inequality holds true for all $\theta > 0$ and $t > 0$:
$$
\E \lc \exp \lp \theta Z_t \rp \rc < \infty \hspace{0.2in} \forall\theta, t > 0. 
$$
In addition,
$$
\lim_{t \to \infty} \dfrac{1}{t} \log \lp \E \lc \exp \lp \theta Z_t \rp \rc \rp = \Psi(\theta),
$$
where $\Psi$ is a function from $(0,\infty)$ to $[0,\infty)$.
\end{lemma}
We now give a rigorous meaning to the quantity $V_t(x)$ by showing that it can be seen as a $L^2-$limit of $V_t^{\ep}(x)$. We also include some exponential bounds which are crucial for the Feynman-Kac representation of \eqref{eq:she-1}.

\begin{proposition}\label{prop:exp_moment_conv}
For $\ep > 0$, $t \geq 0$ and $x \in \R$, let $V_t^{\ep}(x)$ be defined by \eqref{eq:V_t^ep}. Then
\begin{itemize}
\item[(i)] $\lbrace V_t^{\ep} (x); \ep > 0 \rbrace$ is a convergent sequence in $L^2(\Omega \times \hat{\Omega})$. We call its limit $V_t(x)$, where $V_t(x)$ is defined by \eqref{eq:V_t-def}.
\item[(ii)] For all $q \geq 1$ we have 
$$
\lim_{\ep \downarrow 0} \be \otimes \E \lc \lln e^{q V_t^{\ep}} - e^{q V_t} \rrn \rc = 0. 
$$
\end{itemize}
\end{proposition}

\begin{proof}
We divide this proof in several steps.

\noindent
\emph{Step 1: Proof of (i).}
Observe that $V_{t}^{\ep}(x)$ can be written as $\int_0^t W(l_{\ep}(B_r^x-\cdot))dr$, where $W(l_{\ep}(B_r^x - \cdot))$ has to be understood as a Wiener integral conditionally on $B$ (see \eqref{eq:integral_wiener}). In the following we try to find $\lim_{\ep_1,\ep_2 \to 0}\be \otimes \mathbb{E}\left[V_{t}^{\ep_1}(x)V_{t}^{\ep_2}(x)\right]$, which is enough to ensure the $L^2$ convergence of $V_t^{\ep}(x)$. To this aim, we invoke the isometry~\eqref{eq:cov_noise_Fourier} in order to get
\begin{align*}
 \E\otimes \be \left[V_{t}^{\ep_1}(x) V_{t}^{\ep_2}(x)\right] &= \mathbb{E}\left[ \int_0^t \int_0^t W(l_{\ep_1}(B_u^x - \cdot)) W(l_{\ep_2}(B_v^x - \cdot)) du~dv\right]\\
&=\mathbb{E}\int_0^t \int_0^t \int_{\R} \mathcal{F}l_{\ep_1}(B_u^x - \cdot)(\xi)\overline{\mathcal{F}l_{\ep_2}(B_v^x - \cdot)(\xi)} \mu(d\xi)~du~dv.
\end{align*}
Taking into account the expression for $\cf l_{\ep} (B_u^x - \cdot)$ we thus get
\begin{align}\label{a1}
\E\otimes \be \left[V_{t}^{\ep_1}(x) V_{t}^{\ep_2}(x)\right] &= \mathbb{E}\left[ \int_0^t \int_0^t \int_{\R} \cf l (\ep_1\xi) e^{ -\iota \left< \xi, B_u^x \right>} \overline{\cf l(\ep_2 \xi)} e^{ \iota \left<\xi,B_v^x \right>}\mu(d\xi)~du~dv \right]
\notag \\
&=\mathbb{E} \left[ \int_{\R}\left( e^{-\iota\left<\xi,B_u^x-B_v^x\right>}du~dv \right)\cf l (\ep_1 \xi)\overline{\cf l (\ep_2 \xi)} \mu(d\xi) \right].
\end{align}
We can now use the fact that $B_u^x - B_v^x \sim \mathcal{N}(0, v-u)$ to write
\beq\label{eq:cov_V-t^ep}
\E \otimes \be \left[V_{t}^{\ep_1}(x) V_{t}^{\ep_2}(x)\right] =\int_{\R}  \lp \int_{[0,t]^2} \psi_{\ep_1, \ep_2} (u, v;\xi) du dv \rp\mu(d\xi),  
\eeq
where $\psi_{\ep_1, \ep_2}(u,v;\xi)$ is defined by
$$
\psi_{\ep_1, \ep_2} (u,v;\xi) = e^{-\frac{1}{2} {|\xi|}^2 |v-u|} \cf l (\ep_1 \xi)\overline{\cf l (\ep_2 \xi)} .
$$
Moreover, setting $\psi(u,v;\xi) = e^{-\frac{1}{2} {|\xi|}^2 |v-u|}$, it is readily seen that
$$
\lim_{\ep_1, \ep_2 \to 0} \psi_{\ep_1, \ep_2}(u,v;\xi) = \psi(u,v; \xi),\quad \text{ and } \quad \lln \psi_{\ep_1, \ep_2} (u,v ; \xi) \rrn \leq \lln \psi(u,v; \xi) \rrn.
$$
In addition, the reader can check that 
$$
\int_{\R} \int_{[0,t]^2} \psi(u,v; \xi) \,du \,dv \,\mu(d\xi) \leq c\dfrac{\mu(d \xi)}{1+{|\xi|}^2} < \infty.
$$
Therefore, a standard application of the dominated convergence theorem to relation~\eqref{eq:cov_V-t^ep} proves that for every sequence $\ep_n$ converging to zero, $V_{t}^{\ep_n}(x)$ converges in $L^2$ to a limit denoted by $V_{t}(x)$ as mentioned before.

\noindent
\emph{Step 2: Conditional law of $V_t(x)$.}
We will next show that $V_t$ is conditionally Gaussian for all $t \geq 0$ with conditional variance given by
\beq\label{eq:V_tvar}
\be\lc  V_t^2 \rc = \int_{\R} {\lln \int_0^t e^{\imath \xi B_s} ds \rrn}^2 \mu (d \xi). 
\eeq
This will follow from similar calculations as before. First observe that $V_t^{\ep}$ is conditionally Gaussian, with conditional variance given by 
\beq\label{eq:V_t^ep_var}
\be {\lc V_t^{\ep} \rc}^2 = \be \lc  {\lp \int_0^t W(l_{\ep} (B_r^x - \cdot))dr \rp}^2 \rc 
\eeq
The right hand side of \eqref{eq:V_t^ep_var} can be simplified by using the covariance structure of our noise as follows, using the same computations as for \eqref{a1}:
\begin{equation*}
\be {\lc {\lp V_t^{\ep} \rp}^2 \rc} 
= \int_0^t \int_0^t \be \lc W(l_{\ep} (B_r^x - \cdot))W(l_{\ep} (B_s^x - \cdot))\rc dr ds 
= \int_{\R} {\lln \cf l(\ep \xi) \rrn}^2 {\lln \int_0^t e^{\imath \xi B_s}  ds\rrn}^2 \mu(d \xi).
\end{equation*}
Since $V_t$ is the $L^2$ limit of $V_t^{\ep}$ and $L^2$ limits of Gaussian processes remain Gaussian, we now have that conditioned on the Brownian motion, $V_t$ is Gaussian with zero mean and variance given by \eqref{eq:V_tvar}.

\noindent
\emph{Step 3: Exponential moments of $V_t$.}
Our next aim is to show that $V_t$ entertains exponential moments. Specifically we will prove that for all $q > 0$ we have
\beq\label{eq:unconditional_exp_exponential_V_t}
\be \otimes \E \lc e^{q V_t} \rc < \infty.
\eeq
Since we have already shown that $V_t$ is conditionally Gaussian, we have
\beq\label{a2}
\be \lc e^{q V_t} \rc = \exp \lp \dfrac{q^2}{2} \int_{\R} {\lln \int_0^t e^{\imath \xi B_s} ds \rrn}^2 \mu(d \xi) \rp.   
\eeq
Hence, the unconditional expectation of $e^{q V_t}$ is given by
$$
\E \otimes \be \lp e^{q V_t} \rp = \E \lc \exp \lp \dfrac{q^2}{2} \int_{\R} {\lln \int_0^t e^{\imath \xi B_s} ds \rrn}^2  \mu(d \xi)\rp \rc.
$$
%%%%%%%%%%%%%%%%%%%%%%%%%%%%%%%%%%%%%%%%
To see that this quantity is finite let us define the following random variable
$$
Z_t =  \dfrac{1}{t} \int_{\R} {\lln \int_0^t e^{\imath \la B_u} du \rrn}^2 \mu(d \la) .
$$
Observe that we can write:
\begin{align}\label{eq:Z_t/t_subadd-2}
\dfrac{Z_{s+t}}{s+t} &= \dfrac{1}{(s+t)^2} \int_{\R} {\lln \int_0^{s+t} e^{\imath \la B_u} du \rrn}^2\mu(d \la)\nonumber\\
&= \int_{\R} {\lln \dfrac{s}{s+t} {\lp \dfrac{1}{s} \int_0^s e^{\imath \la B_u} du \rp} + \dfrac{t}{s+t} {\lp \dfrac{1}{t} \int_s^{s+t} e^{\imath \la B_u} du \rp} \rrn}^2 \mu(d \la).
\end{align}
Using Jensen's inequality in \eqref{eq:Z_t/t_subadd-2} we now obtain:
\begin{align*}
\dfrac{Z_{s+t}}{s+t} &\leq \int_{\R} \lp  \dfrac{s}{s+t} {\lln \dfrac{1}{s} \int_0^s e^{\imath \la B_u} du \rrn}^2 + \dfrac{t}{s+t} {\lln  \dfrac{1}{t} \int_s^{s+t} e^{\imath \la B_u} du \rrn}^2 \rp \mu(d \la)\\
&= \dfrac{Z_{s} + Z'_t}{s+t} \quad\text{ where }\quad {Z'_t} = \dfrac{1}{t}\int_{\R} {\lln \int_s^{s+t} e^{\imath \la B_u} du \rrn}^2 \mu (d \la). 
\end{align*}
We have thus obtained that $Z$ satisfies the following sub-additive property:
\beq\label{eq:Z_t/t_subadd}
Z_{s+t} \leq Z_s + {Z'_t}. 
\eeq
Moreover, notice that $Z_t'$ above can be written as
$$
Z_t' = \dfrac{1}{t} \int_{\R} {\lln \int_s^{s+t} e^{\imath \la (B_u - B_s)} du \rrn}^2 \mu(d \la),
$$ 
due to the fact that ${| e^{- i \la B_s} |}^2 = 1$. Hence, it is readily checked that $Z_t'$ is independent of $\lcl B_u; 0 \le u \le s \rcl$ and thus also independent of $\lcl Z_u; 0 \le u \le s \rcl$. In addition, ${Z'_t}\stackrel{d}{=} Z_t$. 
Let us now slightly generalize those considerations. Namely, consider a new process $\tilde{Z}$ defined as
\beq\label{eq:tildeZ}
\tilde{Z}_T = \max_{t \leq T} Z_t.
\eeq
It is easily seen that the new process $\tilde{Z}_t$ is also sub-additive in nature. In other words, for all $T_1, T_2 \geq 0$, we have
$$
\tilde{Z}_{T_1 + T_2} \leq \tilde{Z}_{T_1} + \tilde{Z}_{T_2}',
$$
where $\tilde{Z}'_{T_2}$ is independent of $\lbrace \tilde{Z}_t; 0 \leq t \leq T_1 \rbrace$ with $\tilde{Z}_{T_2}'\stackrel{d}{=}\tilde{Z}_{T_2}$.
In addition, since $\tilde{Z}_0 = 0$ and $\tilde{Z}$ has continuous paths, we can apply Lemma~\ref{lem:subadditive-exp} in order to obtain for all $\theta > 0$ and $t > 0$:
$$
\E \lc \exp \lcl \theta \tilde{Z}_t \rcl \rc < \infty,
$$
and as a direct consequence we also have:
$$
\E \lc \exp \lcl \theta {Z}_t \rcl \rc < \infty.
$$
This proves the boundedness of the unconditional expectation of the exponential moments of $V_t$ as expressed in \eqref{eq:unconditional_exp_exponential_V_t}.

\noindent
\emph{Step 4: Conclusion}.
Observe that using the mean value theorem in its integral form and then Cauchy-Schwarz inequality one can write:
\begin{align}\label{eq:ineq-1}
\E \otimes \be \lc \lln e^{qV_t^{\ep}} - e^{q V_t} \rrn \rc &= \E \otimes \be \lc \lln q \lp V_t - V_t^{\ep} \rp \int_0^1 e^{\la q V_t^{\ep} + (1-\la)qV_t} \rrn \rc \nonumber\\
&\leq q {\lp \E \otimes \be \lc {\lln V_t - V_t^{\ep} \rrn}^2 \rc \rp}^{\frac{1}{2}} {\lp \E \otimes \be {\lc {\lln \int_0^1 e^{\la q V_t^{\ep} + (1-\la)qV_t} d\la \rrn}^2 \rc}\rp}^{\frac{1}{2}}.
\end{align}
Using Cauchy-Schwarz inequality on two consecutive occasions separated by Fubini, the object on the right hand side in \eqref{eq:ineq-1} can be further decomposed as:
\begin{align}\label{eq:ineq-2}
\E \otimes \be \lc {\lln \int_0^1 e^{\la q V_t^{\ep} + (1-\la)qV_t} d\la \rrn}^2 \rc &\leq \int_0^1 \E \otimes \be \lc e^{2 \la q V_t^{\ep} + 2(1-\la)q V_t} \rc d\la \nonumber \\
&\leq \int_0^1 {\lp \E \otimes \be \lc e^{4\la q V_t^{\ep}} \rc \rp}^{\frac{1}{2}} {\lp \E \otimes \be \lc e^{4 (1- \la)q V_t} \rc \rp}^{\frac{1}{2}} d\la
\end{align}
Observe from the variance of $V_t^{\ep}$ calculated earlier in \eqref{a2} that
$$
\be \lc e^{q V_t^{\ep}} \rc = \exp \lc \dfrac{q^2}{2} \int_{\R}  e^{- \ep \xi^2}{\lln \int_0^t e^{\imath \xi B_s} ds \rrn}^2 \mu(d \xi) \rc,
$$
and consequently
$$
\E \otimes \be \lc e^{q V_t^{\ep}} \rc \leq \E \otimes \be \lc e^{q V_t} \rc.
$$
Plugging this observation into \eqref{eq:ineq-2} we obtain
\beq\label{eq:ineq-3}
\E \otimes \be \lc {\lln \int_0^1 e^{\la q V_t^{\ep} + (1-\la)qV_t} d\la \rrn}^2 \rc \leq \E \otimes \be \lc e^{4 q V_t} \rc,
\eeq
which is finite by our considerations in Step 3 (see \eqref{eq:unconditional_exp_exponential_V_t}).
Using \eqref{eq:ineq-3} in \eqref{eq:ineq-1} we have
$$
\E \otimes \be \lc \lln e^{qV_t^{\ep}} - e^{q V_t} \rrn \rc \leq q {\lp \E \otimes \be \lc {\lln V_t - V_t^{\ep} \rrn}^2 \rc \rp}^{\frac{1}{2}} {\lp \E \otimes \be \lc e^{4 q V_t} \rc \rp}^{\frac{1}{2}}.
$$
Since $\lbrace V_t^{\ep} (x); \ep > 0 \rbrace$ is a convergent sequence in $L^2(\Omega \times \hat{\Omega})$, our conclusion now follows by taking limits. 
\end{proof}

With the exponential moments of $V_t(x)$ in hand, we can now obtain the announced Feynman-Kac representation of $u$.
\begin{proposition}\label{prop:F-K}
Consider the Gaussian noise $\dot{W}$ defined by \eqref{eq:cov_noise_Fourier}. Let $u$ be the unique solution of equation~\eqref{eq:she-1} with initial condition $u_0(x) = 1$, written in its mild form as:
\beq\label{eq:integral-form-2}
u_t(x) = 1 + \int_0^t p_{t-s} (u_s \dot{W}) ds.
\eeq
Then $u$ can be represented as
\beq\label{eq:f-k}
u_t(x) = \E_x \lc \exp \lp V_{t}(x) \rp \rc, 
\eeq
where $V_t(x)$ is the Feynman-Kac functional defined by \eqref{eq:V_t-def}.
\end{proposition}

\begin{proof}
For $\ep > 0$, let $l_{\ep}$ be the approximation of the identity given in the Introduction. We define a smoothed noise $\dot{W}^{\ep}$ by $\dot{W}^{\ep}= \dot{W} \ast l_{\ep}$, as well as the approximation $u^{\ep}$ of $u$ as the solution of 

\beq\label{eq:mild_form}
u_t^{\ep} (x) = 1+ \int_0^t p_{t-s} (u_s^{\ep} \, \dot{W}^{\ep})ds.
\eeq
Along the same lines as for Proposition~\ref{prop:exi+uni} we can prove that
$$
\lim_{\ep \downarrow 0} u^{\ep} = u \text{ in }\cac_q^{\ka_u, \la, \si} ,
$$
where $\ka_u, \la$ and $q$ are defined in Proposition~\ref{prop:exi+uni}. In addition, since $u^{\ep}$ solves \eqref{eq:mild_form}
in the strong sense, it also admits a Feynman-Kac representation of the form
$$
u_t^{\ep} (x) = \E \lc e^{V_t^{\ep}(x)} \rc,
$$
where $V_t^{\ep}(x)$ is defined by \eqref{eq:V_t^ep}.
For any $p \geq 1$, we are now claiming that for all $T > 0$ we have 
\beq\label{eq:u_t^ep_convergence}
\lim_{\ep \downarrow 0} \sup_{(t,x) \in [0,T] \times \R} \E \lc {\lln u_t^{\ep} (x) - u_t(x) \rrn}^p \rc = 0.
\eeq
In order to get \eqref{eq:u_t^ep_convergence}, notice that 
\begin{align*}
\be \lc {\lln u_t^{\ep} (x) -u_t(x) \rrn}^p \rc &= \be \lc {\lln \E \lc e^{V_t(x)} -e^{V_t^{\ep}(x)} \rc \rrn}^p \rc\\ 
&\leq \be \otimes \E \lc {\lln V_t(x) - V_t^{\ep}(x) \rrn}^p   \lp e^{p V_t(x)} + e^{p V_t^{\ep} (x)} \rp\rc.
\end{align*}
An elementary application of Cauchy-Schwarz inequality and the fact that $V_t(x)$, $V_t^{\ep}(x)$ are conditionally Gaussian yield
\begin{multline*}
\be \lc {\lln u_t^{\ep}(x) - u_t(x) \rrn}^p \rc \\
\leq c_p {\lp \be \otimes \E \lc {\lln V_t(x) - V_t^{\ep}(x) \rrn}^{2} \rc \rp}^{\frac{p}{2}} \lc {\lp \be \otimes \E \lc e^{2p V_t^{\ep} (x)} \rc \rp}^{\frac{1}{2}} +  {\lp \be \otimes \E \lc e^{2p V_t (x)} \rc \rp}^{\frac{1}{2}} \rc.
\end{multline*}
We can now apply directly Proposition~\ref{prop:exp_moment_conv} in order to get
$$
\lim_{\ep \downarrow 0}  \be \lc {\lln u_t^{\ep}(x) - u_t(x) \rrn}^p \rc = 0.
$$
The proof of \eqref{eq:f-k} is now achieved.
\end{proof}

\section{Principal eigenvalues}
Recall that we have shown in Proposition~\ref{prop:F-K} that the unique solution $u$ of our stochastic heat equation~\eqref{eq:integral-form-2} can be written as 
$$
u_t(x) = \E_x \lc \exp(V_t(x)) \rc = \E_x \lc \exp \lp \int_0^t W(\der_{B_s})ds \rp \rc,
$$
where the second identity stems from \eqref{eq:V_t-def}. 

\noindent
Furthermore, $W$ being a homogeneous noise, the asymptotic behavior of $u$ does not depend on the space parameter $x \in \R$. For sake of simplicity we will thus consider $x = 0$ and investigate the quantity 
$$
u_t(0) = \E_0 \lc \exp \lp \int_0^t W(\der_{B_s})ds \rp \rc.
$$

As we will see later on the following equivalence holds true as $t \to \infty$:
\beq\label{eq:solution_asymp_relation}
\E_0 \lc \exp \lp \int_0^t W(\der_{B_s})ds \rp \rc \approx \exp \lp t \la_{\dot{W}} \lp Q_{R_t} \rp \rp
\eeq
for a given region $R_t$ and a principal eigenvalue type quantity $\la_{\dot{W}}$ defined as 
\beq \label{eq:la}
\la_{\dot{W}} (D) = \sup_{g \in \ck (D)} \lcl W(g^2) - \dfrac{1}{2}\int_D {\lln g'(x)\rrn}^2 dx \rcl.
\eeq
In \eqref{eq:la}, $\ck(D)$ is a set of functions defined by 
\beq\label{eq:K}
\ck(D) = \lcl g \in \cs(D): {\|g\|}_2 = 1 \text{ and }g' \in L^2(\R) \rcl,
\eeq
where $\cs(D)$ is the space of infinitely smooth functions that vanish at the boundary of an open domain $D$. Notice that $\ck(D)$ can be seen as a subset of the classical Sobolev space $W^{1,2}(\R)$. In addition, observe that the set $\ck(D)$ is not compact, so that the reader might think that the $\sup$ defining $\la_{\dot{W}}(D)$ in \eqref{eq:la} is ill-defined. However, as we will see in the proof of Proposition~\ref{prop:ub}, our optimization can be reduced by scaling to a compact set $\cg(D)$ defined by
\beq\label{eq:G}
\cg(D) := \lcl g \in \cs(D): {\|g\|}_2^2 + \frac{1}{2}{\|g'\|}_2^2 = 1 \rcl.
\eeq
Before establishing relation \eqref{eq:solution_asymp_relation}, we will try to get some information about the limiting behavior of $\la_{\dot{W}}(D)$ as the size of the box $D$ becomes large. 
\subsection{Basic results}
In this section we establish some Gaussian and analytic results which will be building blocks in the asymptotics \eqref{eq:solution_asymp_relation}. We start by noting that $W(g)$ is a well-defined Gaussian field on the space $\ck(D)$ defined by \eqref{eq:K}.
\begin{lemma}\label{lem:finite_W(g^2)}
Let $g \in \ck(D)$ for any $D \subset \R$. Then 
$$
W(g^2) - \dfrac{1}{2} {\|g\|}_2^2 < \infty \hspace{0.2in} \textnormal{a.s.}
$$
\end{lemma}

\begin{proof}
Note that the variance of $W(g^2)$ is given by
\beq\label{eq:varW(g^2)}
\textnormal{Var} \lc W(g^2) \rc = c_H \int_{\R} {\lln \cf g^2(\xi) \rrn}^2 {\lln \xi \rrn}^{1-2H} d\xi.
\eeq
Also observe that for $g\in\ck(D)$ we have
$$
{\lln \cf g^2(\xi) \rrn} = {\lln \int_{\R} e^{-\imath\xi x} g^2(x) dx \rrn} \leq \int_{\R} {\lln g^2(x) \rrn} dx = 1.
$$
In addition, an elementary integration by parts argument shows that
$$
\int_{\R} e^{-\imath\xi x} g^2(x) dx = - i \int_{\R} \lp \dfrac{1}{\xi} \dfrac{dg^2}{dx} \rp e^{- i \xi x} dx.
$$
Hence for any $\xi \in \R$ and $g \in \ck(D)$ we get
$$
{\lln \cf g^2(\xi) \rrn} \leq {\lln \xi \rrn}^{-1} \int_{\R} {\lln \dfrac{dg^2}{dx} (x)\rrn} dx = 2 {|\xi|}^{-1} \int_{\R} |g(x)||g'(x)| dx \leq 2 {|\xi|}^{-1} {\|g'\|}_2,
$$
where the last inequality follows from Cauchy-Schwarz inequality and observing that ${\|g\|}_2=1$ for $g \in \ck(D)$. Let us now break up the variance in two parts by utilizing the two bounds just established.
\beq\label{eq:cfg^2-mu-L1}
\int_{\R} {\lln \cf g^2(\xi) \rrn}^2 {|\xi|}^{1-2H} d\xi \leq \int_{-1}^{1} {|\xi|}^{1-2H} d\xi + 4{\|g'\|}_2^2 \int_{|\xi|\geq 1} {|\xi|^{-(1+2H)}} d\xi = \dfrac{1}{1-H} + \dfrac{4{\|g'\|}_2^2}{H}
\eeq
We thus get that the variance of $W(g^2)$ is bounded and consequently $W(g^2)$ is finite almost surely. Coupled with the fact that ${\| g' \|}_2 < \infty$ whenever $g$ is an element of $\ck(D)$, we get that
$$
W(g^2) - \dfrac{1}{2} {\|g'\|}_2^2 < \infty \hspace{0.2in} \textnormal{a.s.}
$$
\end{proof}

\begin{remark} 
The following variational quantity will play a prominent role in our limiting results (see also \eqref{eq:ce-0} in Theorem~\ref{thm:Introthm-2}):
\beq\label{eq:ce}
\ce \equiv \sup_{g \in \cg(\R)} \int_{\R} {\lln\int_{\R} e^{\imath \la x}  g^2(x) dx \rrn}^2 {\lln \la \rrn}^{1-2H} d\la , 
\text{ where $\cg$ is defined by \eqref{eq:G}.}
\eeq
The computations of Lemma~\ref{lem:finite_W(g^2)} imply that $\ce$ is a finite quantity. Moreover, if we denote $\ce = \ce(\dot{W})$, then it is easily seen from relation \eqref{eq:varW(g^2)} that 
\beq\label{eq:ce-scaling}
\ce(p\dot{W}) = p^2 \ce(\dot{W})
\eeq
\end{remark}

%\begin{question}
%\hb{We should also prove that $\sup_{g \in \ck(Q_{r})} \lcl W(g^2) - \frac{1}{2} {\|g'\|}_2^2 \rcl <\infty$ for a given box $Q_{r}$. Is this easy to obtain?}
%\end{question}
The first result we need on Gaussian processes is an entropy type bound.
\begin{lemma}\label{lem:2}
Let $\dot{W}$ be the noise defined by \eqref{eq:cov_noise_Fourier}, and recall that $\cg(-\ep, \ep)$ is given by \eqref{eq:G} for all $\ep > 0$. Then we have:
$$
\lim_{\ep \to 0^+} \be \lc \sup_{g \in \cg(-\ep, \ep)} W(g^2) \rc = 0.
$$
\end{lemma}
\begin{proof}
The beginning of the proof is similar to \cite[Lemma 2.2]{chen-ann14}, and we will skip the details for sake of conciseness. Indeed, one can mimic the entropy arguments developed in \cite[Proposition 2.1]{chen-ann14} and show that
$$
\lim_{\der \downarrow 0} \E \sup \lcl W(g^2) ; g \in \cg(Q_1)\text{ and }\E [W(g^2)]^2 \leq \der\rcl = 0.
$$
Then, still following the steps of \cite[Lemma 2.2]{chen-ann14}, it suffices to show that
\beq\label{eq:lem2-1}
\lim_{\ep \downarrow 0} \sup_{g \in \cg(Q_{\ep})} \E {\lc W(g^2) \rc}^2 = 0.
\eeq
To establish \eqref{eq:lem2-1} we use the alternate expression for our covariance function as in \eqref{eq:cov_noise_homSob}, which yields the following expression for all functions $g \in \cg(Q_{\ep})$:
\beq\label{eq:lem2-a}
\E {\lc W(g^2) \rc}^2 = c_H \int_{\R} \int_{\R} \dfrac{{|g^2(x+y) - g^2(x)|}^2}{|y|^{2-2H}} dx dy.
\eeq
Since the domain of any function $g \in \cg(Q_{\ep})$ is contained in $Q_{\ep}$, let us break the right hand side of \eqref{eq:lem2-a} into three parts by integrating over three regions ${\lbrace R_i \rbrace}_{i=1,2,3}$, where
\begin{align*}
R_1 &= \lcl (x,y): |x| \leq \ep, |x+y| \leq \ep \rcl,\\
R_2 &= \lcl (x,y): |x| \leq \ep, |x+y| > \ep \rcl,\\
R_3 &= \lcl (x,y): |x| > \ep, |x+y|  \leq \ep \rcl.
\end{align*}
Consequently, 
$$\E[W(g^2)]^2 = I_{1, \ep} + I_{2, \ep} + I_{3,\ep},$$
where
$$
I_{i,\ep} = c_H \int_{R_i} \dfrac{|g^2(x+y)-g^2(x)|^2}{|y|^{2-2H}} dx dy.
$$
Let us now work  with each integral $I_{i,\ep}$ in succession. In order to upper-bound $I_{1, \ep}$, observe that
$$
{\lln g^2(x+y) - g^2(x) \rrn} = \lln g(x+y) + g(x) \rrn \lln g(x+y) - g(x) \rrn,
$$
and that
\begin{align}\label{eq:lem2-2}
\lln g(x+y) - g(x) \rrn &= \lln \int_x^{x+y} g'(z) dz \rrn\nonumber \\
&\leq \sqrt{\lln \int_x^{x+y} {\lln g'(z) \rrn}^2 dz \rrn |y|} \leq {\|g'\|}_2 \sqrt{|y|},
\end{align}
by an application of the Cauchy-Schwarz inequality. Thus the integrand in $I_{1, \ep}$ can be upper-bounded as follows:
\begin{align}\label{eq:lem2-3}
\dfrac{{\lln g^2(x+y) - g^2(x) \rrn}^2}{|y|^{2-2H}} &= \dfrac{{\lln g(x+y) + g(x) \rrn}^2{\lln g(x+y) - g(x) \rrn}^2}{|y|^{2-2H}}\nonumber\\
&\leq \dfrac{2 \lp g^2(x+y) + g^2(x) \rp {\|g'\|}_2^2 |y|}{|y|^{2-2H}},
\end{align}
where we have used \eqref{eq:lem2-2} and the fact that ${|a+b|}^2 \leq 2(a^2 + b^2)$. Plugging \eqref{eq:lem2-3} in $I_{1,\ep}$ and using the fact that ${\|g'\|}_2^2 \leq 2$ for every $g \in \cg(Q_{\ep})$, we obtain:
\begin{align}\label{eq:lem2-4}
I_{1,\ep} &\leq 4c_H \int_{R_1} \dfrac{g^2(x+y) + g^2(x)}{{|y|}^{1-2H}}dx dy \nonumber \\
&= 4c_H \lc \int_{-\ep}^{\ep} dx \int_{-\ep - x}^{\ep - x} dy~ \dfrac{g^2(x+y)}{|y|^{1-2H}} + \int_{-\ep}^{\ep} dx \int_{-\ep-x}^{\ep-x}dy~ \dfrac{g^2(x)}{|y|^{1-2H}} \rc \nonumber \\
&\leq 4c_H \lc \int_{-\ep}^{\ep} dx \int_{-\ep}^{\ep} dz~ \dfrac{g^2(z)}{|z-x|^{1-2H}} + \int_{-\ep}^{\ep}dx~ g^2(x)\int_{-2\ep}^{2\ep} dy~ \dfrac{1}{|y|^{1-2H}}\rc.
\end{align}
Let us now recall some basic analytic facts taken from \cite[Chapter 4]{adams}: the Sobolev space $W^{1,2}$ is embedded in any $L^k(\R)$ for all $k \geq 2$. More specifically, for all $k \geq 2$ we have
\beq\label{eq:lem2-b}
{\|g\|}_{L^k(\R)} \leq c_k {\|g\|}_{W^{1,2}(\R)},
\eeq
where $c_k$ is a positive constant independent of $g$.

We shall invoke \eqref{eq:lem2-b} in order to bound the first integral in the right hand side of \eqref{eq:lem2-4}. Namely, apply H\"older's inequality with two conjugate numbers $p$ and $q$, which gives 
$$
\int_{(-\ep, \ep)^2} \dfrac{g^2(z)}{{|z-x|}^{1-2H}}dx dz \leq {\lp \int_{(-\ep, \ep)^2} {|g(z)|}^{2p} dx dz \rp}^{\frac{1}{p}} {\lp \int_{(-\ep, \ep)^2} \dfrac{dx dz}{{|z-x|}^{(1-2H)q}} \rp}^{\frac{1}{q}}.
$$ 
We now take a small constant $\der > 0$ and $q = \frac{1-\der}{1-2H}$, which means that $p = \frac{1-\der}{2H-\der}$. Then inequality \eqref{eq:lem2-b} plus some elementary computations show that for $\ep < 1$
$$
\int_{(-\ep, \ep)^2} \dfrac{g^2(z)}{{|z-x|}^{1-2H}} \leq c_{H, \der} {\|g\|}_{W^{1,2}(\R)}^2 \ep \leq 3 c_{H, \der} \ep,
$$
where we resort to the fact that ${\|g\|}_{W^{1,2}(\R)}^2 \leq 3$ whenever $g \in \cg({Q_{\ep}})$ for the last inequality. Using this information in \eqref{eq:lem2-4} {and noting that the second term in \eqref{eq:lem2-4} is bounded thanks to elementary considerations,} we obtain:
\beq\label{eq:lem2-5}
I_{1, \ep} \leq c_H \lp 3c_{H, \der} \ep + {\|g\|}_2^2 {\ep}^{2H} \rp \leq c_{H{,\der}}{{\ep}^{2H}}.
\eeq

Let us now work with $I_{2, \ep}$ and $I_{3, \ep}$. Observe that $I_{2, \ep}$ can be expressed as follows:
\begin{align*}
I_{2, \ep} &= c_H\int_{R_2} \dfrac{{\lln g^2(x+y) -g^2(x) \rrn}^2}{{|y|}^{2-2H}}dx dy \\
&= c_H\int_{R_2} \dfrac{g^4(x)}{|y|^{2-2H}} dx dy\\
&= c_H\int_{-\ep}^{\ep} g^4(x) \lc \int_{-\infty}^{-\ep-x} \dfrac{dy}{|y|^{2-2H}} + \int_{\ep -x}^{\infty} \dfrac{dy}{|y|^{2-2H}}\rc dx\\
&= \dfrac{c_H}{1-2H} \int_{-\ep}^{\ep} g^4(x) \lc \dfrac{1}{(\ep - x)^{1-2H}} + \dfrac{1}{(\ep+x)^{1-2H}}\rc dx.
\end{align*}
We let the patient reader check that the same kind of identity holds for $I_{3, \ep}$. Thus, we find that
\beq\label{eq:lem2-c}
I_{2, \ep} + I_{3, \ep} \leq \dfrac{2c_H}{1-2H} \int_{-\ep}^{\ep} g^4(x) \lc \dfrac{1}{(\ep - x)^{1-2H}} + \dfrac{1}{(\ep + x)^{1-2H}} \rc dx.
\eeq
In order to bound the right hand side of \eqref{eq:lem2-c}, we use the same strategy as for $I_{1, \ep}$. Namely, for $p, q \geq 1$ satisfying $\frac{1}{p} + \frac{1}{q} = 1$, H\"older's inequality imply
$$
I_{1, \ep} + I_{2, \ep} \leq \dfrac{2c_H}{1-2H} {\lp \int_{-\ep}^{\ep} g^{4p}(x)dx \rp}^{\frac{1}{p}} \lc {\lp \int_{-\ep}^{\ep} \dfrac{dx}{(\ep-x)^{(1-2H)q}} \rp}^{\frac{1}{q}}  + {\lp \int_{-\ep}^{\ep} \dfrac{dx}{(\ep + x)^{(1-2H)q}} \rp}^{\frac{1}{q}}\rc.
$$ {As before,} let us now fix {$q=\frac{1-\der}{1-2H}$}.  This implies that the integrals $\int_{-\ep}^{\ep}{(\ep \pm x)^{-(1-2H)q}} dx$ are finite and each is equal to {$c_{\der} \ep^{1-\der}$ for a universal constant $c_{\der}$}. Putting together this information, we find
$$
I_{2, \ep} + I_{3, \ep} \leq {c_{H, \der} {\|g\|}_{4p}^4 \ep^{\frac{\der}{q}}}.
$$
Moreover, a second usage of equation~\eqref{eq:lem2-b} plus the fact that ${\|g\|}_{W^{1,2}(\R)} \leq \sqrt{3}$ yield:
$$
{\|g\|}_{4p} \leq c_p {\|g\|}_{W^{1,2}} \leq c_p \sqrt{3}.
$$
Thus, we obtain:
\beq\label{eq:lem2-6} 
I_{2, \ep} + I_{3, \ep} \leq {c_{H, \der} {\ep^{\frac{\der}{q}}}}.
\eeq
Combining the inequalities \eqref{eq:lem2-5} and \eqref{eq:lem2-6} we find that  
$$
I_{1, \ep} + I_{2, \ep} + I_{3, \ep} \leq {c_{H, \der} \ep^{\nu}}, 
$$
{for a given $\nu > 0$}, uniformly for all $g \in \cg(Q_{\ep})$. Therefore we get 
$$\
\lim_{\ep \downarrow 0} {I_{1,\ep} + I_{2,\ep} + I_{3, \ep}} = 0.
$$
We have thus proved \eqref{eq:lem2-1}. 
\end{proof}

We now introduce the scalings which will be needed in our future computations. 
\begin{notation}\label{not:g_t+h_t}
For fixed $u>0$ and $t > 0$ we introduce a scaling coefficient $h_t$ defined as:
\beq\label{eq:h_t}
h_t = \sqrt{u}(\log t)^{\frac{1}{2(1+H)}}. 
\eeq
Then for each $g \in \cs(\R)$, we define a $L^2(\R)$-rescaled function $g_t$ as follows:
\beq\label{eq:g_t}
g_t(x) = \sqrt{h_t}g(h_t x). 
\eeq
Also, denote by $Q_r$ the open interval $(-r,r)$.
\end{notation}

We now see how to rescale the principal eigenvalues related to $\dot{W}$ in boxes of the form $Q_r$.
\begin{lemma}\label{lem:scaled-eigen}
Let $\dot{W}$ be the Gaussian noise defined by \eqref{eq:cov_noise_Fourier}. For a box $Q_t = (-t, t)$, recall that $\la_{\dot{W}}(Q_t)$ is given by formula \eqref{eq:la}. Then the following relation holds true:
\beq\label{eq:scaled-eigen}
\la_{\dot{W}}(Q_t) = h_t^2 \sup_{g \in \ck(Q_{th_t})} \lcl \frac{1}{h_t^2} W(g_t^2) - \frac{1}{2}\int_{Q_{t{h_t}}} {|g'(x)|}^2 dx \rcl,
\eeq
where the quantities $h_t$ and the function $g_t$ are introduced in Notation~\ref{not:g_t+h_t}.
\end{lemma}
\begin{proof}
Notice that the map $g \mapsto g_t$ when defined from $\ck(Q_{th_t})$ to $\ck(Q_t)$ is a $L^2(\R)$-isomorphism between the two spaces. As a consequence, $\sup_{g \in \ck(Q_t)} A(g) = \sup_{g \in \ck(Q_{th_t})} A(g_t)$ for any general functional $A$ defined on a domain included in $L^{2}(\R)$. Hence,
$$
\la_{\dot{W}}(Q_t) = \sup_{g \in \ck(Q_t)} \lcl W(g^2)-\frac{1}{2}\int_{Q_t} {|g'(x)|}^2 dx \rcl = \sup_{g \in \ck(Q_{th_t})} \lcl W(g_t^2)-\frac{1}{2}\int_{Q_t} {|g_t'(x)|}^2 dx \rcl .
$$
Also, since $g_t'(x) =  {h_t}^{3/2}  g'(h_t x)$, we get
$$\int_{Q_t} {\lln g_t'(x) \rrn}^2 dx = \int_{Q_t} h_t^3 {\lln g'(h_t x) \rrn}^2 dx = h_t^2 \int_{Q_{th_t}} {\lln g'(y) \rrn}^2 dy,$$
where the second identity is due to an elementary change of variable. Consequently,
\begin{align}
\la_{\dot{W}}(Q_t) &= \sup_{g \in \ck(Q_t)} \lcl W(g^2)-\frac{1}{2}\int_{Q_t} {|g'(x)|}^2 dx \rcl \nonumber \\
&= h_t^2 \sup_{g \in \ck(Q_{th_t})} \lcl \frac{1}{h_t^2} W(g_t^2) - \frac{1}{2}\int_{Q_{t{h_t}}} {|g'(x)|}^2 dx \rcl,
\end{align}
which is our claim.
\end{proof}

\begin{remark}\label{rmk:scaling}
One can justify the scaling by $h_t$ given by \eqref{eq:h_t} in the following way: let us start with the rescaled version \eqref{eq:scaled-eigen} of $\la_{\dot{W}} (Q_t)$, which is valid for any weight $h_t$. In addition, we will see in Section~\ref{sec:UB} that the main quantity we should handle in \eqref{eq:scaled-eigen} is the family $\lcl h_t^{-2} W(g_t^2); t \geq 0 \rcl$ and we want this family of Gaussian random variables to remain stochastically bounded in $t$ as $t \to \infty$. Next an elementary computation (see \eqref{eq:si_t^2_bnd} below for more details) reveals that for all $g \in \ck$ we have
\beq\label{eq:si_t,g^2}
\si_{t,g}^2 \equiv \textnormal{Var} \lc {h_t^{-2}} W(g_t^2) \rc = c_{H, g} h_t^{-2(1+H)}.
\eeq
Due to the Gaussian nature of $h_t^{-2} W(g_t^2)$, we thus have (for all $x > 0$)
\beq\label{eq:tail_prob_ht^2W(gt^2)}
\mathbf{P} \lp h_t^{-2} \lln W(g_t^2) \rrn > x \rp  \leq c_1 e^{-{c_2 x^2}/{\si_{t,g}^2}}.
\eeq
A natural way to have the family $\lbrace h_t^{-2} W(g_t^2) ; t \geq 0  \rbrace$ stochastically bounded  is thus to pick the minimal $h_t$ such that one can use Borel-Cantelli in the right hand side of \eqref{eq:tail_prob_ht^2W(gt^2)}. It is readily checked that this is achieved as long as $\si_{t,g}^{-2}$ is of order $\log t$. Recalling the expression~\eqref{eq:si_t,g^2} for $\si_{t,g}^2$, this yields $h_t$ of order ${(\log t)}^{{1}/{2(1+H)}}$. 
\end{remark}

In the following two subsections we explore the long-time asymptotics of $\la_{\dot{W}} (Q_t)$. More precisely, we will try to prove the following:
\beq\label{eq:la_asymptotics}
\lim_{t \to \infty} 
\frac{\la_{\dot{W}}(Q_t)}{(\log t)^{1/(1+H)}} = {\lp 2 c_H \ce \rp}^{{1}/{1+H}} \hspace{0.2in} \textnormal{a.s.} 
\eeq

\subsection{Upper Bound}\label{sec:UB}

In order to get the upper bound part of \eqref{eq:la_asymptotics} we rely on the general idea that principal eigenvalues over a large domain can be essentially bounded by the maximum value among the principal eigenvalues on some sub-domains. See \cite[Proposition~1]{gartner-konig} where this result is proved when the potential is defined pointwise. In \cite{chen-ann14} the same result is stated to be true for generalized functions as well. We start with an elementary lemma whose proof is very similar to the aforementioned references.

\begin{lemma}\label{lem:eigen-gartner}
Let $r > 0$. There exists a non-negative continuous function $\Phi(x)$ on $\R$ whose support is contained in the $1-$neighborhood of the grid $2r\Z$, such that for any $R>r$ and any generalized function $\xi$,
\beq\label{eq:eigen-1}
\la_{\xi - \Phi^y}(Q_R) \leq \max_{z \in 2r\Z \cap Q_R} \la_{\xi}(z+Q_{r+1}), \qquad \text{for all }y\in Q_r,
\eeq
where $\Phi^y(x) = \Phi(x+y)$.
In addition $\Phi(x)$ is periodic with period $2r$, namely
\beq\label{eq:eigen-1.5}
\Phi(x+2rz) = \Phi(x),~~~x \in \R,~z\in\Z,
\eeq
and there is a constant $K>0$ independent of $r$ such that 
\beq\label{eq:eigen-2}
\dfrac{1}{2r} \int_{Q_r} \Phi(x)dx \leq \dfrac{K}{r}.
\eeq
\end{lemma}
We now show how to split the upper bound for the principal eigenvalue $\la_{\dot{W}}(Q_t)$ into small subsets.
\begin{lemma}
Let $W$ be the noise defined by \eqref{eq:b} and $Q_t = (-t, t)$. We consider the principal eigenvalue $\la_{\dot{W}}(Q_t)$ given by \eqref{eq:la}. Recalling that $h_t$ is given by \eqref{eq:h_t}, the following inequality holds true:
\beq \label{eq:la_X_z-ineq}
\la_{\dot{W}}(Q_t) \leq h_t^2 \lp \dfrac{K}{r} + \max_{z \in 2r\Z \cap Q_{th_t}} X_z(t) \rp,
\eeq
where the random field $\lcl X_z(t); z \in 2r\Z, t \geq 0 \rcl$ is defined by:
\beq\label{eq:X-z}
X_z(t) = \sup_{g \in \ck(z+Q_{r+1})} \lcl \dfrac{W(g_t^2)}{h_t^2} - \dfrac{1}{2}\int_{Q_{th_t}} {|g'(x)|}^2 dx \rcl.
\eeq
In \eqref{eq:X-z}, the set $\ck(z+Q_{r+1})$ is given by \eqref{eq:K} and the function $g_t$ is defined by \eqref{eq:g_t}.  
\end{lemma}
\begin{proof}
Let $\lcl W_t(\psi), \psi \in \cd(\R) \rcl$ be the generalized Gaussian field defined as 
$W_t(\psi) = W(\hpsi_t)$ where $\hpsi_t(x) = h_t\psi(h_t x)$. 
Then with the definition \eqref{eq:g_t} of $g_t$ in mind, notice that $W(g_t^2)= W_t(g^2)$. Thus invoking Lemma~\ref{lem:scaled-eigen} we have:
\begin{align*}
&\dfrac{1}{h_t^2} \la_{\dot{W}} (Q_t) = \sup_{g \in \ck(Q_{th_t})} \lcl \dfrac{1}{h_t^2} W(g_t^2) - \dfrac{1}{2}\int_{Q_{th_t}} {|g'(x)|}^2 dx \rcl\\
&= \sup_{g \in \ck(Q_{th_t})} \lcl \dfrac{1}{h_t^2}W_t(g^2)- \dfrac{1}{2}\int_{Q_{th_t}} {|g'(x)|}^2 dx\rcl\\
&= \sup_{g \in \ck(Q_{th_t})} \lcl \lla \dfrac{1}{h_t^2}\dot{W_t} - \dfrac{1}{2r} \int_{Q_r} \Phi^y dy, g^2 \rra + \lla \dfrac{1}{2r}\int_{Q_r}\Phi^y(x)dy, g^2 \rra -\dfrac{1}{2}\int_{Q_{th_t}} {|g'(x)|}^2 dx \rcl,
\end{align*}
where $\langle \dot{W}_t, g^2 \rangle$ is understood in the distribution sense. Hence inequality \eqref{eq:eigen-2} and the fact that $\langle g^2, \mathbf{1} \rangle = 1$ if $g \in \ck(Q_{th_t})$ yields
$$ \dfrac{1}{h_t^2} \la_{\dot{W}} (Q_t) \leq {\dfrac{K}{r} + \sup_{g \in \ck(Q_{th_t})} \lcl \lla \dfrac{1}{h_t^2}\dot{W_t} - \dfrac{1}{2r}\int_{Q_r}\Phi^y dy, g^2 \rra -\dfrac{1}{2}\int_{Q_{th_t}} {|g'(x)|}^2 dx \rcl}.$$
Therefore bounding $\sup \int$ by $\int \sup$ and invoking the definition \eqref{eq:la} of the principal eigenvalue, we end up with:
\begin{align*}
&\dfrac{1}{h_t^2} \la_{\dot{W}} (Q_t) \leq {\dfrac{K}{r} + \dfrac{1}{2r} \int_{Q_r} \sup_{g \in \ck(Q_{th_t})} \lcl \lla \dfrac{1}{h_t^2}\dot{W_t} - \Phi^y, g^2 \rra - \dfrac{1}{2}\int_{Q_{th_t}} {|g'(x)|}^2 dx\rcl dy}\\
&\leq \dfrac{K}{r} + \dfrac{1}{2r} \int_{Q_r} \la_{\frac{\dot{W_t}}{h_t^2} -\Phi^y} (Q_{th_t}) dy. 
\end{align*}
We can now resort to \eqref{eq:eigen-1} in order to get:
$$
\dfrac{1}{h_t^2} \la_{\dot{W}} (Q_t) \leq \dfrac{K}{r} + \max_{z \in 2r\Z \cap Q_{th_t}} \la_{\frac{\dot{W_t}}{h_t^2}} \lp z + Q_{r+1}\rp.
$$
Recall again that $\dot{W_t}$ is defined by $W_t(\psi) = W(\hpsi_t)$ where $\hpsi_t(x)=h_{t}\psi(h_{t}x)$.
Thus we have 
$$
\dfrac{1}{h_t^2} \la_{\dot{W}} (Q_t) \leq \dfrac{K}{r} + \max_{z \in 2r\Z \cap Q_{th_t}} X_z(t),
$$
where the random fields $\lcl X_z(t); z \in 2r\Z, t \geq 0 \rcl$ are defined by \eqref{eq:X-z}.
Our claim \eqref{eq:la_X_z-ineq} is thus easily deduced. 
\end{proof}
We are ready to state the desired lower bound on our principal eigenvalue.
\begin{proposition}\label{prop:ub}
Let $\la_{\dot{W}}(Q_t)$ be the principal eigenvalue of the random operator $\frac{1}{2}\Delta + \dot{W}$ over the restricted space {$\ck(Q_{t})$} of functions having compact support on $(-t,t)$, defined by~\eqref{eq:G}. Then the following limit holds:
$$
\limsup_{t \to \infty} \dfrac{\la_{\dot{W}}(Q_t)}{(\log t)^{\frac{1}{1+H}}} \leq {\lp 2 c_H \ce \rp}^{\frac{1}{1+H}},  \hspace{0.2in} \textnormal{a.s.} 
$$
where we recall that $\ce$ is defined by \eqref{eq:ce}.
\end{proposition}
\begin{proof}
We shall rely on relation \eqref{eq:la_X_z-ineq} and bound $\max_{z \in 2r\Z \cap Q_{th_t}} X_z(t)$ thanks to Gaussian entropy methods. We divide the proof in several steps.

\noindent
\emph{Step 1: Reduction to a Gaussian supremum.}
By homogeneity of the Gaussian field $\lbrace W(\phi); \phi \in \cd(\R) \rbrace$, the random variables ${\lcl X_z(t) \rcl}_{ z \in 2r\Z \cap Q_{th_t}}$ are identically distributed. Consequently we have

\begin{align}\label{eq:X_z-X_0}
\bp \lc \max_{z \in 2r\Z \cap Q_{th_t}} X_z(t) > 1 \rc &\leq \#\lcl 2r\Z \cap Q_{th_t} \rcl\bp\lc X_0(t) > 1 \rc \nonumber \\
&\leq \lp \dfrac{th_t}{r} \rp \bp \lc X_0(t) > 1 \rc.
\end{align}
Recalling the definition \eqref{eq:X-z} of $X_0(t)$, we thus get
\begin{multline}\label{eq:Xz_ineq_1}
\bp \lc \max_{z \in 2r\Z \cap Q_{th_t}} X_z(t) > 1 \rc \\
 \leq {\lp \dfrac{th_t}{r} \rp} \bp \lc \sup_{g \in \ck(Q_{r+1})} \lcl \dfrac{1}{h_t^2}W(g_t^2) - \dfrac{1}{2} \int_{Q_{th_t}} {\lln g'(x) \rrn}^2 dx\rcl > 1 \rc.
\end{multline}
Notice that in \eqref{eq:Xz_ineq_1} the Gaussian supremum for the family $\lp W(g_t^2) \rp$ is taken over the set $\ck$ given by \eqref{eq:K}. However, this set is not compact, which is not suitable for Gaussian computations (see e.g. the discussion after \cite[Lemma~1.3.1]{adler-taylor}). In the following steps we will reduce our computations to an optimization over a compact set of the form $\cg$ (see equation~\eqref{eq:G}). To this aim,
for any $g \in \ck(Q_{r+1})$, set
$$
\phi = \frac{g}{{\sqrt{1+\frac{1}{2}{\|g'\|}_2^2}}}.
$$
Notice that since ${\|g\|}_2 = 1$, we have 
$$
\phi \in \cg(Q_{r+1}), \quad \text{and} \quad \phi_t = \frac{g_t}{\sqrt{1+\frac{1}{2}{\|g'\|}_2^2}},
$$ 
where the notation $\phi_{t}$ is given by \eqref{eq:g_t}.
Therefore the following rough estimate holds true for the parameter $h_t$ defined by \eqref{eq:h_t}: 
$$
\dfrac{1}{h_t^2} W(\phi_t^2) \leq \dfrac{1}{h_t^2} \sup_{f \in \cg(Q_{r+1})} W(f_t^2).
$$
Moreover, recalling that $\phi_t^2 = {\lp 1 + \frac{1}{2}{\|g'\|}_2^2 \rp}^{-1} g_t^2$, we find
$$
\dfrac{1}{h_t^2} W(g_t^2) \leq \dfrac{\lp 1+ \dfrac{1}{2} {\|g'\|}_2^2 \rp}{h_t^2} \sup_{f \in \cg(Q_{r+1})} W(f_t^2).
$$
Thus, subtracting ${\|g'\|}_2^2$ on both sides of the above equation we get the following relation for all $g \in \ck(Q_{r+1})$:
$$
\lp \dfrac{1}{h_t^2} W(g_t^2) - \dfrac{1}{2} \int_{Q_{r+1}}{\lln g'(x)\rrn}^2 dx\rp \leq \dfrac{\lp 1+\dfrac{1}{2}{\|g'\|}_2^2 \rp}{h_t^2} \sup_{f \in \cg(Q_{r+1})} W(f_t^2) - \dfrac{1}{2}{\|g'\|}_2^2
$$
Taking supremum over $g \in \ck\lp{Q_{r+1}} \rp$, this yields
$$
X_0(t) \leq \sup_{g \in \ck(Q_{r+1})} 
\lcl \dfrac{\lp 1+\dfrac{1}{2}{\|g'\|}_2^2 \rp}{h_t^2} \sup_{f \in \cg(Q_{r+1})} W(f_t^2)  - \dfrac{1}{2}{\|g'\|}_2^2 \rcl.
$$
Consequently, if $X_0(t) \geq 1$, we also have 
$$
\sup_{g \in \ck(Q_{r+1})} \lp \dfrac{\sup_{f \in \cg(Q_{r+1})} W(f_t^2)}{h_t^2}  - 1\rp \lp 1+\dfrac{1}{2} {\|g'\|}_2^2 \rp \geq 0,
$$
or otherwise stated:
$$
\lc \sup_{f \in \cg(Q_{r+1})} \dfrac{W(f_t^2)}{h_t^2} - 1 \rc \sup_{g \in \ck(Q_{r+1})} \lp 1+ \dfrac{1}{2}{\|g'\|}_2^2 \rp \geq 0.
$$
It is readily checked that the above condition is met iff $\sup_{f \in \cg(Q_{r+1})} W(f_t^2) \geq h_t^2$. Summarizing, we have shown that 
$$
\lcl X_0(t) \geq 1 \rcl \subset \lcl \sup_{f \in \cg(Q_{r+1})} W(f_t^2) \geq h_t^2 \rcl, 
$$
which implies 
\beq\label{eq:prob_gauss_sup}
\bp \lp X_0(t) \geq 1 \rp \leq \bp \lc \sup_{g \in \cg(Q_{r+1})} W(g_t^2) \geq h_t^2 \rc.
\eeq
We are now reduced to the desired sup over a compact set.

\noindent
\emph{Step 2: Gaussian concentration.}
We now evaluate the right hand side of \eqref{eq:prob_gauss_sup} by standard Gaussian supremum estimates. Namely, some elementary scaling arguments show that for each $g \in \cg\lp{Q_{r+1}}\rp$, 
$$
{\lp{1 + \dfrac{\lp h_t^2 - 1 \rp}{2} {\|g'\|}_2^2}\rp}^{-1/2} {g_t} \in \cg(Q_{(r+1)/h_t}).
$$
Moreover by the linearity of Gaussian fields and due to the fact that ${\|g'\|}_2^2 \leq 2$ whenever $g \in \cg(Q_{r+1})$, we get 
\beq \label{eq:exp_sup_W}
\be\lc \sup_{g \in \cg(Q_{r+1})} W(g_t^2)\rc 
\leq 
h_t^2 \, \be \lc \sup_{f \in \cg(Q_{(r+1)/h_t})} W(f^2)\rc 
\equiv h_t^2 \delta_t .
\eeq
In addition, Lemma~\ref{lem:2} asserts that $\lim_{t \to \infty} \delta_t = 0$ (notice that the fact of working on a box with finite size $r+1$ is crucial for this step). We are now in a position to invoke Borell-TIS concentration inequality for Gaussian fields (See \cite[Theorem 2.1.2]{adler-taylor}) and our inequality~\eqref{eq:exp_sup_W}, which yields
\begin{align}\label{eq:gauss_ineq}
\bp& \lc \sup_{g \in \cg(Q_{r+1})} W(g_t^2) \geq h_t^2 \rc \nonumber  \\
&=\bp \lc \sup_{g \in \cg(Q_{r+1})} W(g_t^2) - \be \lp \sup_{g \in \cg(Q_{r+1})} W(g_t^2) \rp \geq h_t^2 (1 -\delta_t) \rc \nonumber \\
& \leq \exp \lc - \dfrac{h_t^4 {\lp 1-\delta_t \rp}^2}{2\si_t^2} \rc,
\end{align}
where $\si_t^2$ is a parameter defined by $\si_t^2 = \sup_{g \in \cg(Q_{r+1})} \textnormal{Var}\lc W(g_t^2) \rc$.

We now find an upper bound for the term $\si_t^2$ in \eqref{eq:gauss_ineq}. This is achieved as follows:
Owing to the definition \eqref{eq:cov_noise_Fourier} of the covariance of $W$, we have
\begin{align*}
\si_t^2 &= c_H \sup_{g \in \cg(Q_{r+1})} \int_{\R} {\lln \cf g_t^2(\xi) \rrn}^2 {|\la|}^{1-2H} d\la\\
&= c_H \sup_{g \in \cg(Q_{r+1})} \int_{\R} {\lln \int_{\R} e^{\imath \la x} g_t^2(x) dx \rrn}^2 {|\la|}^{1-2H} d\la .
\end{align*}
Therefore, recalling the definition \eqref{eq:g_t} of $g_{t}$ and invoking some easy scaling arguments we obtain:
\begin{align}\label{eq:si_t^2_bnd}
\si_t^2 &= c_H h_t^{2-2H}  \sup_{g \in \cg(Q_{r+1})} \int_{\R} {\lln \int_{\R} e^{\imath\la x} g^2(x) dx \rrn}^2 {|\la|}^{1-2H} d\la \nonumber \\
&\leq c_H h_t^{2-2H}  \sup_{g \in \cg(\R)} \int_{\R} {\lln \int_{\R} e^{\imath\la x} g^2(x) dx \rrn}^2 {|\la|}^{1-2H} d\la = c_H h_t^{2-2H} \ce,
\end{align}
where we recall that $\ce$ is a finite quantity according to \eqref{eq:ce}.
We can plug our upper bound~\eqref{eq:si_t^2_bnd} for the item $\si_t^2$ in \eqref{eq:gauss_ineq} and replace $h_t$ by its value $\sqrt{u} {\lp \log t \rp}^{1/(2(1+H))}$. We end up with:
\begin{equation*}
\bp \lc \sup_{g \in \cg(Q_{r+1})} W(g_t^2) \geq h_t^2 \rc \leq \exp\lp -\dfrac{(1-\delta_t)^2 u^{1+H}}{2 \ce c_H} \log t \rp .
\end{equation*}

We wish the series $\sum_k \bp \lp \sup_{g \in \cg(Q_{r+1})} W(g_{2^k}^2) \geq h_{2^k}^2 \rp $ to be convergent. To this aim, owing to the fact that $\lim_{t \to \infty} \delta_t = 0$, for $t$ sufficiently large we get
\beq\label{eq:gauss_final_bnd}
\bp \lp \sup_{g \in \cg(Q_{r+1})} W(g_t^2) \geq h_t^2 \rp \leq \exp\lc -\lp 1+\nu \rp \log t \rc
=\frac{1}{t^{1+\nu}},
\eeq
where $\nu > 0$ is a small enough constant, provided the following condition is met:
\beq\label{eq:ub_u_cond}
u > {\lp 2 c_H  \ce \rp}^{1/(1+H)}.
\eeq
Here we highlight the fact that $t^{-(1+\nu)}$ is obtained in the right hand side of \eqref{eq:gauss_final_bnd}. This exponent lead to our choice of scaling by $h_t = \sqrt{u} {(\log t)}^{\frac{1}{2(1+H)}}$ in our computations (see Remark~\ref{rmk:scaling}). 

\noindent
\emph{Step 3: Conclusion.} Now, we summarize our steps so far. Thanks to \eqref{eq:X_z-X_0}, \eqref{eq:prob_gauss_sup} and \eqref{eq:gauss_final_bnd} we have
\begin{align*}
\bp \lc \max_{z \in 2r\Z \cap Q_{th_t}} X_z(t) \geq 1 \rc &\leq {\lp\dfrac{th_t}{r}\rp} \bp \lc X_0(t) \geq 1 \rc\\
&\leq {\lp\dfrac{th_t}{r}\rp} \bp \lc \sup_{g \in \cg(Q_{r+1})} W(g_t^2) \geq h_t^2 \rc\\
&\leq {\lp\dfrac{th_t}{r}\rp} \exp\lp -(1+\nu)\log t \rp = \dfrac{h_t}{r} \dfrac{1}{t^{\nu}}.
\end{align*}
Take the sequence $t_k = 2^k$. Then we have 
$$\bp \lc \max_{z \in 2r\Z \cap Q_{t_k h_{t_k}}} X_z(t_k) \geq 1 \rc \leq \frac{\sqrt{u}}{r} \frac{(\log t_k)^{\frac{1}{2(1+H)}}}{t_k^{\nu}}= \frac{\sqrt{u}}{r}(k \log 2)^{\frac{1}{2(1+H)}} 2^{-k\nu},$$ 
and the right hand side of the above inequality is the general term of a convergent series. By Borell-Cantelli Lemma, we thus have
\beq \label{eq:X_z(t)_upper-bnd}
\limsup_{k \to \infty} \max_{z \in 2r\Z \cap Q_{t_k h_{t_k}}} X_z(t_k) < 1,\hspace{0.2in} \textnormal{a.s.}
\eeq
We now draw conclusions on the principal eigenvalue itself. Indeed, from \eqref{eq:la_X_z-ineq} and \eqref{eq:X_z(t)_upper-bnd}, it is readily checked that
$$
\limsup_{k \to \infty} \dfrac{\la_{\dot{W}}(Q_{t_k})} {{\lp \log t_k \rp}^{\frac{1}{(1+H)}}} < \lp \dfrac{K}{r} + 1\rp u, \hspace{0.2in} \textnormal{a.s.}
$$
Thus some elementary monotonicity arguments show that
\beq\label{eq:upper_bnd_r}
\limsup_{t \to \infty} \dfrac{\la_{\dot{W}}(Q_{t})}{ {\lp \log t \rp}^{\frac{1}{(1+H)}}} < \lp \dfrac{K}{r} + 1\rp u \hspace{0.2in} \textnormal{a.s.}
\eeq
Since the constant $K$ in \eqref{eq:upper_bnd_r} is independent of $r$, and $r$ can be arbitrarily large, we also get
$$
\limsup_{t \to \infty} \dfrac{\la_{\dot{W}}(Q_{t})}{ {\lp \log t \rp}^{\frac{1}{(1+H)}}} \leq  u \hspace{0.2in} \textnormal{a.s.}
$$

Eventually recall that we had to impose the condition \eqref{eq:ub_u_cond} on $u$. However $u$ can be taken as close as we wish to the value ${\lp 2 c_H \ce \rp}^{\frac{1}{1+H}}$. As a consequence we get
$$
\limsup_{t \to \infty} \dfrac{\la_{\dot{W}}(Q_{t})}{ {\lp \log t \rp}^{\frac{1}{(1+H)}}} \leq {\lp 2 c_H \ce \rp}^{{1}/{1+H}} \hspace{0.2in} \textnormal{a.s.} 
$$
\end{proof}
\subsection{Lower Bound}
This section is devoted to a lower bound counterpart of Proposition~\ref{prop:ub}. We start by a lemma asserting that $\la_{\dot{W}}(Q_t)$ cannot get too small with respect to an order of magnitude of $h_t^2$.
\begin{lemma}\label{lem:la-W-ineq}
Let $\la_{\dot{W}}(Q_t)$ be the principal eigenvalue of the random operator $\frac{1}{2}\Delta + \dot{W}$ over the restricted space $\ck(Q_{t})$ of functions having compact support on $(-t,t)$. Then we have the following upper bound:
\beq\label{eq:la-W-ineq}
\bp \lc \la_{\dot{W}} (Q_t) \leq h_t^2 \rc \leq \bp \lc \sup_{g \in \cg(Q_{th_t})} \dfrac{W(g_t^2)}{h_t^2} \leq 1 \rc.
\eeq
\end{lemma}
\begin{proof}
Observe that from \eqref{eq:scaled-eigen}, 
\beq\label{eq:eigen-lb}
\bp \lp \la_{\dot{W}}(Q_t) \leq h_t^2 \rp = \bp \lc \sup_{g \in \ck(Q_{th_t})} \lcl \dfrac{W(g_t^2)}{h_t^2} - \dfrac{1}{2} \int_{Q_{th_t}} {\lln g'(x) \rrn}^2 dx \rcl \leq 1 \rc.
\eeq
Moreover, 
\beq\label{eq:prob-1}
\bp \lc \sup_{g \in \ck(Q_{th_t})} \lcl \dfrac{W(g_t^2)}{h_t^2} - \dfrac{1}{2} \int_{Q_{th_t}} {\lln g'(x) \rrn}^2 dx \rcl \leq 1 \rc \leq \bp \lc \sup_{g \in \cg(Q_{th_t})} \dfrac{W(g_t^2)}{h_t^2} \leq 1 \rc.
\eeq
This is proved similarly to our considerations in Step~1 of the proof of Proposition~\ref{prop:ub}, details are included here for sake of clarity.

Namely, in order to prove \eqref{eq:prob-1}, notice that for $g \in  \cg\lp Q_{th_t}\rp$, we have $\phi = \frac{g}{{\|g\|}_2} \in \ck(Q_{th_t})$. Consequently
$$
\dfrac{W(\phi_t^2)}{h_t^2} - \dfrac{1}{2} \int_{Q_{th_t}} {\lln \phi'(x) \rrn}^2 dx 
\leq 
\sup_{f \in \ck(Q_{th_t})} \lcl \dfrac{W(f_t^2)}{h_t^2} - \dfrac{1}{2} \int_{Q_{th_t}} {\lln f'(x) \rrn}^2 dx \rcl .
$$
Thus the bound
\beq \label{eq:sup_W(g_t^2)_extra}
\sup_{f \in \ck(Q_{th_t})} \lcl \dfrac{W(f_t^2)}{h_t^2} - \dfrac{1}{2} \int_{Q_{th_t}} {\lln f'(x) \rrn}^2 dx \rcl \leq 1
\eeq
implies, still with $\phi=g/\|g\|_{2}$,
$$
\dfrac{W(\phi_t^2)}{h_t^2} - \dfrac{1}{2} \int_{Q_{th_t}} {\lln \phi'(x) \rrn}^2 dx  \leq 1.
$$
This in turn gives the following inequality when we write down $\phi$ in terms of $g$:
$$
\dfrac{W(g_t^2)}{h_t^2} - \dfrac{1}{2} \int_{Q_{th_t}} {\lln g'(x) \rrn}^2 dx \leq {\|g\|}_2^2.
$$
Therefore we have obtained, for every $g \in \cg(Q_{th_t})$,
$$
\dfrac{W(g_t^2)}{h_t^2} \leq {\|g\|}_2^2 + \dfrac{1}{2} {\|g'\|}_2^2 = 1,
$$
where the last equality follows from the fact that $g \in \cg(Q_{th_t})$. Taking supremum and recalling that we have assumed \eqref{eq:sup_W(g_t^2)_extra}, we get
$$
\lcl \sup_{g \in \ck(Q_{th_t})} \lcl \dfrac{W(g_t^2)}{h_t^2} - \dfrac{1}{2} \int_{Q_{th_t}} {\lln g'(x) \rrn}^2 dx \rcl \leq 1 \rcl \subset \lcl \sup_{g \in \cg(Q_{th_t})} \dfrac{W(g_t^2)}{h_t^2} \leq 1\rcl.
$$
Thus, \eqref{eq:prob-1} is proved and \eqref{eq:eigen-lb} can be further reduced to 
$$
\bp \lc \la_{\dot{W}} (Q_t) \leq h_t^2 \rc \leq \bp \lc \sup_{g \in \cg(Q_{th_t})} \dfrac{W(g_t^2)}{h_t^2} \leq 1 \rc,
$$
which proves our result~\eqref{eq:la-W-ineq}.
\end{proof}

Our next lemma is a general bound for Gaussian vectors with nontrivial covariance structure. It is borrowed from \cite[Lemma~4.2]{chen-ann14} and will be used in a discretization procedure which is part of our strategy for the lower bound on $\la_{\dot{W}}(Q_t)$. %In the following lemma we state an inequality to bound the maximum entry of a Gaussian vector. 
\begin{lemma}\label{lem:gauss_vec_ineq}
Let $(\xi_1, \dots, \xi_n)$ be a mean-zero Gaussian vector with identically distributed components. Write $R = \max_{i \neq j}\lln \textnormal{Cov} (\xi_i, \xi_j) \rrn$ and assume that $\textnormal{Var}(\xi_1) \geq 2R$. Then for any $A, B > 0$, the following inequality holds true:
$$
\bp \lc \max_{k \leq n} \xi_k  \leq A \rc \leq  {\lp \bp \lc \xi_1 \leq \sqrt{\dfrac{2R+\textnormal{Var}(\xi_1)}{\textnormal{Var}(\xi_1)}} ( A + B ) \rc \rp}^n + \bp\lc U \geq \dfrac{B}{\sqrt{2R}} \rc
$$
where $U$ is a standard normal random variable.
\end{lemma}
We can now state our lower bound on the principal eigenvalue $\la_{\dot{W}}(Q_t)$. 
\begin{proposition}\label{prop:lb}
Under the same conditions as for Lemma~\ref{lem:la-W-ineq}, the following lower bound is fulfilled:
$$
\liminf_{t \to \infty}  \dfrac{\la_{\dot{W}}(Q_{t})}{ {\lp \log t \rp}^{\frac{1}{(1+H)}}} \geq {\lp 2 c_H \ce \rp}^{\frac{1}{1+H}} \hspace{0.2in} \textnormal{a.s.}
$$
\end{proposition}
\begin{proof}
We divide the proof in several steps.

\noindent
\emph{Step 1: Reduction to a discrete Gaussian supremum.}
Let the constant $r > 0$ be fixed but arbitrary and set $\cn_t = 2r\Z \cap Q_{t-r}$. When $t$ is large enough (namely $t>r$ and $h_t > 1$), it is readily checked that $h_tz+Q_r \subset Q_{th_t}$ for each $z \in \cn_t.$ Hence,
$$
\sup_{g \in \cg(Q_{th_t})} W(g_t^2) \geq \max_{z \in \cn_t} \sup_{g \in \cg(h_t z + Q_{r})} W(g_t^2).
$$ 
and thus owing to \eqref{eq:la-W-ineq},
$$
\bp \lc \la_{\dot{W}} (Q_t) \leq h_t^2 \rc \leq \bp\lp \max_{z \in \cn_t} \sup_{g \in \cg(h_t z + Q_{r})} W(g_t^2) \leq h_t^2 \rp.
$$
For any $g \in \cg(Q_r)$ and $z \in \cn_t$, notice that $g^z(\cdot) \equiv g(\cdot - h_t z) \in \cg(h_t z+ Q_r)$.
Hence $\max_{z \in \cn_t} \sup_{g \in \cg(h_t z + Q_{r})} W(g_t^2) \geq \max_{z \in \cn_t} W({(g_t^z)}^2)$, for any $g \in \cg(Q_r).$  The consequent inequality is therefore:
\beq\label{eq:la_W(g^z)_ineq}
\bp \lc \la_{\dot{W}} (Q_t) \leq h_t^2 \rc \leq \bp\lp \max_{z \in \cn_t} W({(g_t^z)}^2) \leq h_t^2 \rp.
\eeq
for any given (but arbitrary) $g \in \cg(Q_r)$.

\noindent
\emph{Step 2: Control of covariance.}
For ease of presentation let us denote $W((g_t^z)^2)$ by $\xi_z(t)$. We will try to control the covariance $\textnormal{Cov}(\xi_z(t), \xi_{z'}(t))$ for $z,z' \in \cn_t$ in order to show that the assumptions of Lemma~\ref{lem:gauss_vec_ineq} are met. First notice that $\cf((g^z)_t^2)$ can be also expressed as:
$$
\cf\lp (g^z)_t^2 \rp(\xi) = \int_{\R}e^{-\imath\xi x}  {\lcl g_t^z(x) \rcl}^2 dx
= \int_{\R}e^{-\imath\xi x} {\lcl \sqrt{h_t} g^z(h_t x) \rcl}^2 dx
= \int_{\R} e^{-\imath \xi x}h_t g^2(h_t (x - z)) dx.
$$
Therefore, with change of variable $s=h_t(x-z)$ we get:
\beq\label{eq:fourier}
\cf\lp (g^z)_t^2 \rp(\xi)
= e^{-\imath \xi z} \int_{\R} e^{-\imath \xi \frac{s}{h_t}} g^2(s) ds
= e^{-\imath\xi z} \cf g^2\lp\dfrac{\xi}{h_t}\rp
\eeq
Hence, the covariance of the random field $\xi_z(t)$ is given by
\begin{align*}
\textnormal{Cov}\lp \xi_z(t), \xi_{z'}(t) \rp &= \int_{\R} \cf (g^z)_t^2(\xi) \overline{\cf(g^{z'})_t^2 (\xi)} \mu (d\xi)\\
& = c_H \int_{\R} e^{\imath \xi(z-z')} {\lln \cf g^2\lp\dfrac{\xi}{h_t}\rp \rrn}^2 {\lln \xi \rrn}^{1-2H} d\xi,
\end{align*}
where the last equality follows by using \eqref{eq:fourier} and by plugging in the value of $\mu$. Changing variable, we can rewrite the covariance as
\beq\label{eq:cov_xi}
\textnormal{Cov}\lp \xi_z(t), \xi_{z'}(t) \rp = c_H h_t^{2(1-H)} \int_{\R} e^{\imath h_t u(z-z')} {\lln \cf g^2(u) \rrn}^2 {|u|}^{1-2H} du.
\eeq
In particular, we have 
\beq\label{eq:var_xi}
\textnormal{Var} \lp \xi_0(t) \rp = c_H h_t^{2(1-H)} \si^2(g)
\eeq
where $\si^2(g) = \int_{\R} {\lln \cf g^2(u) \rrn}^2 {|u|}^{1-2H} du$, which is a finite quantity when $g \in \cg(Q_r)$ according to \eqref{eq:cfg^2-mu-L1}.

Recall that $h_t = \sqrt{u} {\lp  \log t\rp}^{\frac{1}{2(1+H)}}$ and thus $h_t \to \infty$ as $t \to \infty$. In addition, $h_{t}|z-z'|\ge 2h_{t} r$ uniformly for $z\neq z'$ in $\cn_{t}$. Also observe again that $g \in \cg(Q_r)$ and hence $G(u) = {\lln \cf g^2(u) \rrn}^2 {|u|}^{1-2H}$ is in $L^1$ thanks to \eqref{eq:cfg^2-mu-L1}. By Riemann-Lebesgue lemma, we get the following assertion uniformly for $z\neq z'$ in $\cn_{t}$:
$$
\lim_{t \to \infty} \int_{\R} e^{\imath h_t u(z-z')} {\lln \cf g^2(u) \rrn}^2 {|u|}^{1-2H} du = 0.
$$
Therefore, plugging this information into \eqref{eq:cov_xi}, we end up with
\beq\label{eq:R_t-asymp}
R_t = \max_{\substack{{z, z' \in \cn_t}\\{z\neq z'}}} \lln \textnormal{Cov} \lp \xi_z(t), \xi_{z'}(t) \rp \rrn = o(h_t^{2(1-H)}).
\eeq
Furthermore observe that \eqref{eq:var_xi} implies $\lim_{t \to \infty} [\textnormal{Var} (\xi_0(t)) / {h_t^{2(1-H)}} ]  = c_H \si^2(g) > 0$. Thus we also get $\textnormal{Var} (\xi_0(t)) \geq 2R_t$ for $t$ sufficiently large. Summarizing our considerations so far, we have proved that the family $\lcl \xi_z(t); z \in \cn_t \rcl$ satisfies the conditions of Lemma~\ref{lem:gauss_vec_ineq} if $t$ is large enough. We now introduce an additional parameter $v > 0$ (to be chosen small enough later on) and we resort to Lemma~\ref{lem:gauss_vec_ineq} with $A=h_t^2$ and $B=vh_t^2$ in order to write:
\beq\label{eq:max_gauss_ineq}
\bp \lc \max_{z \in \cn_t} \xi_z(t)  \leq h_t^2\rc \leq V_t^1 + V_t^2 
\eeq
where $R_t$ is defined by \eqref{eq:R_t-asymp}, and
$$
V_t^1 = {\lp \bp\lc \xi_0(t) \leq (1+v)h_t^2 \sqrt{\dfrac{2R_t + \textnormal{Var}(\xi_0(t))}{\textnormal{Var} (\xi_0(t))}} \rc \rp}^{\lln\cn_t\rrn}, \quad V_t^2 = \bp \lc  U \geq \dfrac{v h_t^2}{\sqrt{2R_t}}\rc.
$$
We now bound these two terms separately. 

\noindent
\emph{Step 3: Bound on $V_t^2$.}
First, we bound the term $V_t^2$ on the right hand side in \eqref{eq:max_gauss_ineq}. By a classical bound on the normal tail probabilities:
\beq\label{eq:gaussian_tail_1}
\bp \lc U \geq \dfrac{v h_t^2}{\sqrt{2R_t}} \rc \leq \dfrac{1}{\sqrt{2 \pi}} \dfrac{\sqrt{2R_t}}{v h_t^2} \exp\lp -\dfrac{v^2 h_t^4}{4R_t} \rp.
\eeq
Since by \eqref{eq:R_t-asymp}, $\frac{R_t}{h_t^{2(1-H)}} \to 0$ as $t \to \infty$ we have for $t$ sufficiently large
\beq\label{eq:V_t^2-ineq1}
\dfrac{1}{\sqrt{2 \pi}} \dfrac{\sqrt{2R_t}}{v h_t^2} < 1.
\eeq
As for the term inside the exponential in \eqref{eq:gaussian_tail_1}, observe that (recall $h_t = \sqrt{u} {\lp \log t\rp}^{\frac{1}{2(1+H)}}$ again)
$$
\frac{v^2 h_t^4}{4 R_t} = \frac{v^2 h_t^{2(1+H)}}{4} \frac{h_t^{2(1-H)}}{R_t} = \frac{v^2 u^{1+H} }{4} \frac{h_t^{2(1-H)}}{R_t}\log t,
$$ 
and that for $t$ sufficiently large 
\beq\label{eq:V_t^2-ineq-2}
\dfrac{R_t}{h_t^{2(1-H)}} \dfrac{4}{v^2 u^{1+H}} < \dfrac{1}{2}.
\eeq
Plugging \eqref{eq:V_t^2-ineq1} and \eqref{eq:V_t^2-ineq-2} into \eqref{eq:gaussian_tail_1}, for $t$ sufficiently large we have
\beq\label{eq:max_gauss_secondterm}
V_t^2 = \bp \lc U \geq \dfrac{v h_t^2}{\sqrt{2R_t}} \rc \leq \dfrac{1}{\sqrt{2\pi}} \dfrac{\sqrt{2R_t}}{v h_t^2} \exp\lp -\dfrac{v^2 h_t^4}{4R_t} \rp \leq \exp\lp -2 \log t \rp = \dfrac{1}{t^2}.
\eeq

\noindent
\emph{Step 4: Bound on $V_t^1$.}
Let us now bound $V_t^1$ in \eqref{eq:max_gauss_ineq}, which can be written as
\beq
V_t^1 = {\lp \bp \lc \dfrac{\xi_0(t)}{\sqrt{\textnormal{Var}{\lp\xi_0(t) \rp} }} \leq \dfrac{(1+v)h_t^2 \sqrt{2R_t + \textnormal{Var}(\xi_0(t))}}{\textnormal{Var}(\xi_0(t))} \rc \rp}^{\lln \cn_t \rrn}.
\eeq
Therefore according to relation \eqref{eq:var_xi}, we have 
$$
V_t^1 = {\lp \bp \lc U \leq l_t \rc \rp}^{\lln \cn_t \rrn} 
$$
where we have set
$$
l_t = \dfrac{(1+v) h_t^{1+H}}{c_H \si^2(g)} \sqrt{2 \dfrac{R_t}{h_t^{2(1-H)}} + c_H \si^2(g)}.
$$ 
In order to work with this last term we can rewrite it as:
\begin{align}\label{eq:V_t^1-1}
V_t^1 = {\lp\bp\lc U \leq l_t\rc \rp}^{\lln \cn_t \rrn} &= {\lp 1 - \bp\lc U > l_t \rc \rp}^{\lln \cn_t \rrn}\nonumber \\ 
&= {\lc \exp \lp \dfrac{1}{\bp [U > l_t]} \log\lp 1 - \bp[U>l_t] \rp \rp \rc}^{\lln \cn_t \rrn \bp[U > l_t]}.
\end{align}
Owing to \eqref{eq:R_t-asymp} and the fact that $\lim_{t \to \infty} h_t = \infty$, it is easily checked that $\lim_{t \to \infty} l_t = \infty$. Therefore, 
\beq\label{eq:e^{-1}}
\lim_{t \to \infty} \exp\lp \dfrac{1}{\bp[U>l_t]} \log \lp 1-\bp[U>l_t] \rp \rp = e^{-1}
\eeq
Let us now concentrate on $\lln \cn_t \rrn \bp[U > l_t]$ in \eqref{eq:V_t^1-1}. We use the following elementary facts about $l_t$:
\begin{itemize}
\item[(i)] $\lim_{t \to \infty} \frac{R_t}{h_t^{2(1-H)}} = 0$, and thus
\beq\label{eq:l_t-asymp}
\lim_{t \to \infty} \lp l_t - c_{v,g} h_t^{1+H} \rp = 0, \text{ where }c_{v,g} = \frac{(1+v)}{c_H^{1/2} \si(g)}.
\eeq

\item[(ii)] Since $\lim_{t \to \infty} l_t = \infty$, we have $\lim_{t \to \infty} l_t e^{l_t^2/2} \bp[U > l_t] = \frac{1}{\sqrt{2 \pi}}$.
\end{itemize}
Using this information it is easy to see that 
$$
\bp[U > l_t] \sim \dfrac{e^{-\frac{l_t^2}{2}}}{l_t} \sim \dfrac{\exp\lp -\frac{c_{v,g}^2}{2} h_t^{2(1+H)} \rp}{\sqrt{2 \pi}c_{v,g} h_t^{1+H}}
$$
With the expression $h_t = \sqrt{u} {\lp \log(t) \rp}^{\frac{1}{2(1+H)}}$ in mind, this yields 
$$
\bp [U > l_t] \sim \dfrac{\exp\lp -\frac{c_{v,g}^2 u^{1+H}}{2} \log(t) \rp}{\sqrt{2 \pi}{c_{v,g}\lc u^{1+H} \log(t) \rc }^{1/2}}
$$
and thus
$$
\bp[U > l_t] \sim \dfrac{1}{\sqrt{2 \pi}c_{v,g} {\lc u^{1+H} \log(t)\rc}^{1/2} t^{\frac{c_{v,g}^2 u^{1+H}}{2}}}
$$
In addition, we also have $\lln \cn_t \rrn \sim \frac{t}{r}$, which yields
$$
{\lln \cn_t \rrn} \bp [U > l_t] \sim \dfrac{t^{1-\frac{c_{v,g}^2 u^{1+H}}{2}}}{ \sqrt{2 \pi}c_{v,g} {\lc u^{1+H} \log(t) \rc}^{1/2} r}
$$
Recall now that $c_{v,g}$ is given by expression \eqref{eq:l_t-asymp}. Hence, choosing $v$ small enough and provided 
\beq\label{eq:u-upper-bound}
u < {\lp 2 c_H \si^2(g) \rp}^{{1}/{1+H}},
\eeq
we can check that $\frac{c_{v,g}^2 u^{1+H}}{2} < 1$. Consequently, for $t$ large enough we obtain:
\beq\label{eq:t^beta}
{\lln \cn_t \rrn} \bp [U > l_t] \geq t^{\beta}, \text{ for some }\beta>0.
\eeq
Plugging \eqref{eq:e^{-1}} and \eqref{eq:t^beta} into \eqref{eq:V_t^1-1}, we get the following relation for $t$ sufficiently large: 
\beq\label{eq:max_gauss_firstterm}
V_t^1 = {\lp \bp\lc \xi_0(t) \leq (1+v)h_t^2 \sqrt{\dfrac{2R_t + \textnormal{Var}(\xi_0(t))}{\textnormal{Var} (\xi_0(t))}} \rc \rp}^{\lln \cn_t \rrn} \leq e^{-t^{\beta}}.
\eeq

\noindent
\emph{Step 5: Conclusion.}
Reporting inequalities \eqref{eq:max_gauss_secondterm} and \eqref{eq:max_gauss_firstterm} into \eqref{eq:max_gauss_ineq}, we find:
$$
\bp \lc \max_{z \in \cn_t} W((g^z)_t^2) \leq h_t^2 \rc \leq \frac{1}{t^2} + \frac{1}{e^{t^\beta}},
$$
for some $\beta > 0$ whenever 
$$
u < {\lp 2c_H \si^2(g) \rp}^{\frac{1}{1+H}}.
$$
Now appealing to \eqref{eq:la_W(g^z)_ineq} and using Borel-Cantelli Lemma:
$$
\liminf_{k \to \infty} h_{t_k}^{-2}\la_{\dot{W}}(Q_{t_k}) > 1 \hspace{0.2in}\textnormal{a.s.}
$$
for some increasing sequence $t_k$ of integers. Thus the expression $h_t = \sqrt{u} {(\log t)}^{\frac{1}{2(1+H)}}$ and some elementary monotonicity arguments show that
$$
\liminf_{t \to \infty} \la_{\dot{W}} (Q_t) (\log t)^{-1/(1+H)} > u \hspace{0.2in}\textnormal{a.s.}
$$
Now thanks to \eqref{eq:u-upper-bound} and taking $u \uparrow {\lp 2 c_H \si^2(g) \rp}^{\frac{1}{1+H}}$, we have for every $g \in \cg(Q_r)$,
$$
\liminf_{t \to \infty} \la_{\dot{W}}(Q_t)(\log t)^{-1/(1+H)}  \geq {\lp 2 c_H \si^2(g) \rp}^{\frac{1}{1+H}} \hspace{0.2in} {\textnormal{a.s.}}
$$
Recall that $\ce = \sup_{g \in \cg(\R)} \si^2(g)$. Hence taking supremum over $g \in \cg(Q_r)$ and letting $r \to \infty$ gives the needed lower bound
$$
\liminf_{t \to \infty} \la_{\dot{W}}(Q_t)(\log t)^{-1/(1+H)}  \geq {\lp 2 c_H \ce \rp}^{\frac{1}{1+H}} \hspace{0.2in} \textnormal{a.s.}
$$
\end{proof}

\section{Lyapounov exponent}
In this section we will combine the Feynman-Kac representation of $u$ and our preliminary study of the principal eigenvalue $\la_{\dot{W}} (Q_t)$ in order to get the logarithmic behavior of $u_t(x)$, achieving the proof of our main Theorem~\ref{thm:Introthm-4}. 
\subsection{Preliminary Results.}
Recall that $V^{\ep}$ is defined by \eqref{eq:V_t^ep}, and observe that one can also write
$$
V_{t}^{\ep}(x) = \int_0^t \dot{W}^{\ep} (B_s) ds, 
$$
where $\dot{W}^{\ep}$ is the regularized noise given by 
\beq\label{eq:noise-regularised}
\dot{W}^{\ep} (x) = \int_{\R} l_{\ep}(x-y)W(dy).
\eeq

The following lemma will allow us to extend the domains on which principal eigenvalues are computed. The interested reader is referred to \cite[Lemma 2.2]{chen-ann14} for a proof.
\begin{lemma}\label{lem:-1}
Let $W$ be the Gaussian noise whose covariance is defined by \eqref{eq:b}, and let $\dot{W}^{\ep}$ be defined by relation \eqref{eq:noise-regularised}. For a bounded measurable set $D \subset \R$ we write $D_{\ep} = (-\ep, \ep)+D$ and define for positive $\theta$ the eigenvalue type quantity $\la_{\theta \dot{W}}^{+} (D)$ by:
$$
\la_{\theta \dot{W}}^{+} (D) : = \lim_{\ep \downarrow 0} \la_{\theta\dot{W}} (D_{\ep}) 
$$
Then $\la_{\theta \dot{W}^{\ep}}(D)$ is bounded as follows:
$$ 
\la_{\theta \dot{W}} (D) \leq \liminf_{\ep \downarrow 0} \la_{\theta \dot{W}^{\ep}} (D) \leq \limsup_{\ep \downarrow 0} \la_{\theta \dot{W}^{\ep}}(D) \leq \la_{\theta \dot{W}}^{+} (D) \hspace{0.2in}\textnormal{a.s.}
$$
\end{lemma}

The second lemma below is a first relation between Feynman-Kac representations of equation~\eqref{eq:she-1} and principal eigenvalues. It is stated for a general potential $\xi$ which is pointwise defined but not necessarily bounded.

\begin{lemma}\label{lem:0}
Let $\xi:\R \mapsto {\R}$ be a potential, not necessarily bounded. Let $\tau_D$ be the stopping time defined by $\tau_D = \inf \lcl{t\geq 0: B_t \notin D}\rcl$ for a measurable bounded set $D\subset \R$. Then the following inequalities hold where $\la_{\xi}(D)$ is defined similarly to \eqref{eq:la}:
\item[(i)]We have:
\beq\label{eq:lem0}
\int_{D} \E_x \lc \exp\lcl \int_0^t \xi(B_s) ds \rcl \mathbf{1}_{\lcl \tau_D \geq t \rcl} \rc dx \leq {\lln D \rrn} \exp\lcl t \la_{\xi} \lp D \rp \rcl
\eeq

\item[(ii)]For any $\al$, $\beta > 1$ satisfying $\frac{1}{\al}+\frac{1}{\beta}=1$ and $\la_{(\beta/\al) \xi} (D) < \infty$ we have for $0 < \der < t$:
\begin{multline}\label{eq:lem0(ii)}
\int_D \E_x \lc \exp \lcl \int_0^t \xi(B_s) ds \rcl \mathbf{1}_{\lcl \tau_D \geq t \rcl} \rc dx \geq (2\pi)^{\al/2} {\der}^{1/2} t^{\al/(2 \beta)} {|D|}^{-2\al/\beta} \times \\
\exp {\lcl -\der (\al/\beta) \la_{(\beta/\al)\xi} (D) \rcl} \exp\lcl \al(t+\der)\la_{\al^{-1} \xi} (D) \rcl. 
\end{multline}
\end{lemma}

\begin{proof}
The proof of \eqref{eq:lem0} relies on classical Feynman-Kac representations of semigroups. Namely if $T_t g$ is the semigroup on $L^2(D)$ defined by
\beq\label{eq:T_t}
T_t g(x) = \E_x \lc \exp \lcl \int_0^t \xi(B_s) ds \rcl g(B_t) \mathbf{1}_{\lcl \tau_D \geq t \rcl}  \rc, ~t \geq 0, x \in D,
\eeq
it can be shown that the generator $A$ of $T_t$ admits a Dirichlet form defined by 
$$
\lla g, Ag \rra = \int_{D} \xi(x) g^2(x) dx - \dfrac{1}{2} \int_{D} {\lln \nabla g(x) \rrn}^2 dx.
$$
One can prove that 
$$
\la_0 \equiv \sup_{\substack{g \in \cd(A)\\ \|g\| = 1}} \lla g, Ag\rra = \sup_{g \in \ck(D)} \lla g, Ag\rra = \la_{\xi}(D).
$$
Then \eqref{eq:lem0} is obtained thanks to some spectral representation techniques. The reader is referred to \cite{Chen-Book} for further details and to \cite{Chen-Poisson} for the lower bound \eqref{eq:lem0(ii)}.
\end{proof}

The following lemma holds as a consequence of the Markov property for the Brownian motion $B$, and will yield a second relation between Lyapounov exponent and our principal eigenvalue. It is borrowed from \cite[Section 4]{Chen-Poisson}.
\begin{lemma}\label{lem:1}
Let $\xi$ be a potential from $\R$ to $\R$ and $D$ be a measurable bounded set. Let $0< \der < t$ and assume $0 \in D$. Let $\frac{1}{\al} + \frac{1}{\beta} = 1$.
\begin{itemize}
\item[(i)] The following upper bound holds true:
\begin{align*}
\E_0 \lc \exp\lcl \int_0^t \xi(B_s) ds \rcl \mathbf{1}_{\lcl \tau_D \geq t \rcl} \rc &\leq {\lp \E_0 \exp {\lcl \beta \int_0^{\der} \xi(B_s) ds \rcl} \rp}^{1/\beta} \times \\ 
&{\lp \dfrac{1}{(2\pi\der)^{d/2}} \int_D \E_x {\lc \exp {\lcl \al \int_0^{t-\der} \xi(B_s)ds \rcl} \mathbf{1}_{\lcl \tau_D \geq t-\der \rcl} \rc} dx \rp}^{1/\al}.
\end{align*}

\noindent
\item[(ii)] We also have the corresponding lower bound:
\begin{align*}
\E_0 \lc \exp {\lcl \int_0^t \xi(B_s)ds \rcl} \rc &\geq {\lp \E_0 \exp {\lcl -\dfrac{\beta}{\al} \int_0^{\der} \xi(B_s) ds \rcl} \rp}^{-\al/\beta} \times\\
&{\lp \int_D p_{\der}(x) \E_x {\lc \exp \lcl \dfrac{1}{\al} \int_0^{t-\der} \xi(B_s) ds \rcl \mathbf{1}_{\lcl \tau_D \geq t-\der \rcl} \rc} dx \rp}^{\al},
\end{align*}
where we recall that $p_{\der}$ designates the heat kernel in $\R$ (see Notations in the Introduction).
\end{itemize}
\end{lemma}

\subsection{Upper Bound.}
We can now apply the preliminary results on exponential functionals of $B$ recalled in the last section, in order to get a first comparison between $\log(u_t(0))$ and the principal eigenvalue $\la_{\dot{W}}(Q_t)$. The logarithmic asymptotic behavior of $u_t(x)$ can be upper bounded thanks to the following result.
\begin{proposition}\label{prop:UB}
Let $\lbrace u_t(x); t\geq 0, x \in \R \rbrace$ be the field defined by \eqref{eq:f-k}. Then the following holds:
\beq\label{eq:UB-result}
\limsup_{t \to \infty}\dfrac{\frac{1}{t}\log(u_t(0))}{\la_{\dot{W}} (Q_t)} \leq 1,
\eeq
where $Q_t = (-t, t)$ and $\la_{\dot{W}}(D)$ is defined by \eqref{eq:la} for a domain $D$.
\end{proposition}
\begin{proof}

\noindent
\emph{Step 1: Decomposition of $u_t(0)$.}
To implement the upper bound \eqref{eq:UB-result}, let us introduce a constant $M$ to be specified later on and for $k \geq 1$ let $R_k$ be defined by 
\beq\label{eq:R_k}
R_k = {\lcl Mt(\log t)^{\frac{1}{2(1+H)}} \rcl}^k.
\eeq
Also recall that, according to \eqref{eq:f-k}, we have $u_t(x) = \E_x \lc \exp \lp V_t(x) \rp \rc$, where $V_t(x)$ and $V_t^{\ep}(x)$ are defined by \eqref{eq:V_t-def}. We now set $V_t(0) = V_t$ and $V_t^{\ep}(0) = V_t^{\ep}$. With these notations in hand, we decompose $u_t(0)$ as:
\begin{align}\label{eq:UB-u_t-decomp}
u_t(0) &= \E_0 \lc \exp\lp V_t \rp \rc\nonumber \\ &= \E_0 \lc \exp\lp V_t \rp \mathbf{1}_{ \lcl \tau_{Q_{R_1}} \geq t \rcl} \rc
+ \sum_{k=1}^{\infty} \E_0 \lc \exp\lp V_t \rp \mathbf{1}_{\lcl \tau_{Q_{R_k}} < t \leq \tau_{Q_{R_{k+1}}} \rcl} \rc
\end{align}

In order to upper bound the terms in our decomposition \eqref{eq:UB-u_t-decomp}, we apply H\"older's inequality to each term in the sum. We get:
\beq\label{eq:UB-u_t-decomp-shorter}
u_t(0) \leq U_{t,0} + \sum_{k=1}^{\infty} \bp_0(\tau_{Q_{R_k}} < t)^{1/2} U_{t,k}
\eeq
where 
\beq\label{eq:U_t,0}
U_{t,0} = \E_0 \lc e^{V_t} \mathbf{1}_{\lcl \tau_{Q_{R_1}} \geq t \rcl} \rc,
\eeq
and for $k \geq 1$
\beq\label{eq:U_t,k}
U_{t,k} = \E_0^{1/2} \lc e^{2V_t} \mathbf{1}_{\lcl \tau_{Q_{R_{k+1}}} \geq t \rcl} \rc.
\eeq
We will now bound the terms $U_{t,k}$ separately. 

\noindent
\emph{Step 2: Regularization.}
Let us replace the quantities $V_t$ by $V_t^{\ep}$ in the definition of $U_{t,k}$ for $k \geq 0$. The corresponding random variables are denoted by $U_{t,k}^{\ep}$. We start by getting a uniform bound for $U_{t,0}^{\ep}$. Namely using Lemma~\ref{lem:1}(i) write
\begin{align}\label{eq:UB-u_t-decomp-part1}
&U_{t,0}^{\ep} = \E_0 \lc   e^{V_{t}^{\ep}} \mathbf{1}_{\lcl \tau_{Q_{R_1}} \geq t \rcl} \rc = \E_0 \lc \exp \lp \int_0^t \dot{W}^{\ep} (B_r) dr \rp \mathbf{1}_{\lcl \tau_{Q_{R_1}} \geq t \rcl} \rc \nonumber \\
& \leq {\lp \E_0 \lc \exp(q \int_0^1 \dot{W}^{\ep} (B_s) ds) \rc \rp}^{1/q} {\lp \dfrac{1}{\sqrt{2\pi}} \int_{Q_{R_1}} \E_x \lc \exp\lp p \int_{0}^{t-1}\dot{W}^{\ep} (B_s) ds \rp \mathbf{1}_{\lcl\tau_{Q_{R_1}}\geq t-1 \rcl} \rc dx \rp}^{1/p}\nonumber \\
&= {\lp \E_0 \lc e^{q V_1^{\ep}} \rc \rp}^{1/q} {\lp \dfrac{1}{\sqrt{2\pi}} \int_{Q_{R_1}} \E_x \lc e^{pV_{t-1}^{\ep}(x)} \mathbf{1}_{\lcl\tau_{Q_{R_1}} \geq {t-1}\rcl} \rc  dx \rp}^{1/p}
\end{align}

\noindent
We can now apply Lemma~\ref{lem:0}(i) to the right hand side of the above equation. This yields
$$
\int_{Q_{R_1}} \E_x \lc e^{p V_{t-1}^{\ep}(x)} \mathbf{1}_{\lcl \tau_{Q_{R_1}}\geq t-1 \rcl} \rc dx \leq \lln Q_{R_1} \rrn \exp \lc (t-1) \la_{p \dot{W}^{\ep}} (Q_{R_1}) \rc
$$

\noindent
Computing the volume $\lln Q_{R_1} \rrn$ and plugging into \eqref{eq:UB-u_t-decomp-part1} we obtain:
\begin{equation}\label{eq:ineq-regularised}
U_{t,0}^{\ep} \leq {\lp \E_0 \lc \exp\lp q V_1^{\ep} \rp \rc \rp}^{\frac{1}{q}} {\lp \dfrac{2R_1^2}{\pi} \rp}^{\frac{1}{2p}}
 \exp \lc \dfrac{(t-1)}{p} \la_{p \dot{W}^{\ep}} (Q_{R_1}) \rc.
\end{equation}
We now take limits in equation~\eqref{eq:ineq-regularised}. In order to handle the left hand side of \eqref{eq:ineq-regularised}, we observe that the random variable $\E_0 \lc e^{V_t^{\ep}} \rc$ converges in any $L^q(\Omega)$ to $\E_0 \lc e^{V_t} \rc$ thanks to Proposition~\ref{prop:exp_moment_conv}. Therefore for all $q \geq 0$, an easy application of H\"older's inequality shows that 
$$
L^q(\Omega) - \lim_{\ep \downarrow 0} \E_0 \lc e^{V_t^{\ep}} \mathbf{1}_{\lcl \tau_{Q_{R_1}} \geq t \rcl} \rc = \E_0 \lc e^{V_t} \mathbf{1}_{\lcl \tau_{Q_{R_1}} \geq t \rcl} \rc,
$$
where $L^q(\Omega)$ is the space of $L^q$ random variables on $(\Omega, \cf, \bp)$, see Notation~\ref{not:FK-1}. It follows that there exists a subsequence $\lcl {\ep}_n ; n \geq 0 \rcl$ such that $\bp-\textnormal{a.s.}$ we have 
$$
\lim_{n \to \infty} \E_0 \lc e^{V_t^{\ep_n}} \mathbf{1}_{\lcl \tau_{Q_{R_1}} \geq t \rcl} \rc = \E_0 \lc e^{V_t}\mathbf{1}_{\lcl \tau_{Q_{R_1}} \geq t \rcl} \rc,
$$
and in particular 
$$
\E_0 \lc e^{V_t} \mathbf{1}_{\lcl \tau_{Q_{R_1}} \geq t\rcl} \rc \leq \liminf_{\ep \downarrow 0} \E_0 \lc e^{V_t^{\ep}} \mathbf{1}_{\lcl \tau_{Q_{R_1}} \geq t \rcl} \rc.
$$
The same kind of arguments applied to the right hand side of \eqref{eq:ineq-regularised} yield the following relation (recall that $U_{t,0} = \E_0 [ e^{V_t} \mathbf{1}_{\lbrace \tau_{Q_{R_1}} \geq  t \rbrace} ]$ according to \eqref{eq:U_t,0}):
\beq\label{eq:UB-2}
U_{t,0} \leq {\lp \E_0 \lc e^{q V_1} \rc \rp}^{\frac{1}{q}} {\lp \dfrac{2R_1^2}{\pi} \rp}^{\frac{1}{2p}} \exp {\lc \dfrac{(t-1)}{p} \la_{p \dot{W}}^{+} \lp Q_{R_1} \rp \rc}.
\eeq

We can proceed similarly in order to bound the terms $U_{t,k}$ in \eqref{eq:U_t,k}. Indeed applying Cauchy-Schwarz inequality and following the same steps as for \eqref{eq:UB-u_t-decomp-part1}-\eqref{eq:ineq-regularised} we get, for all $k \geq 1$
\beq\label{eq:UB-3}
\E_0 \lc \exp(2 V_t) \mathbf{1}_{\lcl \tau_{Q_{R_{k+1}}} \geq t \rcl} \rc \leq {\lp \E_0 \lc \exp(4 V_1) \rc \rp}^{\frac{1}{2}} {\lp \dfrac{2R_{k+1}^2}{\pi} \rp}^{\frac{1}{4}} \exp \lc \frac{(t-1)}{2} \la_{4\dot{W}}^{+} \lp Q_{R_{k+1}} \rp \rc.
\eeq

\noindent
Consequently plugging \eqref{eq:UB-2} and \eqref{eq:UB-3} into \eqref{eq:UB-u_t-decomp-shorter}, we end up with:
\begin{equation}\label{eq:UB-4}
u_t(0) \leq a_{1,p,q} \mathcal{M}_{p,t} + a_2 \mathcal{R}_{t},
\end{equation}
where 
\beq\label{eq:M_t,R_t}
\mathcal{M}_{p,t} =  \exp\lp \dfrac{(t-1)}{p} \la_{p\dot{W}}\lp Q_{R_1} \rp \rp \text{and}~\mathcal{R}_t = \sum_{k=1}^{\infty} \al_k \exp \lp \dfrac{(t-1)}{4} \la_{4 \dot{W}}^{+} \lp Q_{R_{k+1}} \rp\rp,
\eeq
and where we also recall that $R_k$ is defined by $\eqref{eq:R_k}$, and the constants $a_{1,p,q}$ and $a_2$ are given by 
\begin{align*}
a_{1,p,q} = {\lp{\dfrac{2R_1^2}{\pi}}\rp}^{\frac{1}{2p}} {\lp \E \lc \exp(q V_1) \rc \rp}^{\frac{1}{q}}, \hspace{0.2in}
a_2 = {\lp \E_0 \lc \exp(4V_1) \rc \rp}^{\frac{1}{4}}.
\end{align*}
In \eqref{eq:UB-4}, the constants $\al_k$ for $k \geq 1$ are also defined by:
\beq\label{eq:al_k}
\al_k = {\lp \bP_0 \lp \tau_{Q_{R_k}} < t \rp \rp}^{\frac{1}{2}} {\lp \dfrac{2 R_{k+1}^2}{\pi} \rp}^{\frac{1}{8}}.
\eeq
We will now treat the terms in \eqref{eq:UB-4} separately.

\noindent
\emph{Step 3: Bound on $u_t(0)$.}
Let us first bound the constants $\al_k$ in \eqref{eq:al_k}. To this aim, we can invoke the reflection principle for Brownian motions (see e.g. \cite[section 2.6]{Karatzas}), which asserts that 
$$
\bP_0 \lp \tau_{Q_{R_k}} < t \rp \leq \dfrac{4}{\sqrt{2 \pi t}} \int_{R_k}^{\infty} e^{-\frac{y^2}{2t}} dy \leq \dfrac{4 \sqrt{t}}{ \sqrt{2 \pi} R_k} e^{-\frac{R_k^2}{2t}}. 
$$
Furthermore when $t$ is large enough, it is readily checked from the expression \eqref{eq:R_k} of $R_k$ that
$$
\dfrac{4 \sqrt{t}}{\sqrt{2 \pi} R_k} \leq 1,
$$
uniformly in $k \geq 1$. Therefore we get 
$$
\bP_0 \lp \tau_{Q_{R_k}} < t \rp \leq e^{- \frac{R_k^2}{2t}}.
$$
Plugging this inequality in equation~\eqref{eq:al_k} and designating by $c$ a universal constant which can change from line to line, we get
\beq\label{eq:al_k-bnd}
\al_k \leq c R_{k+1}^{\frac{1}{4}} e^{-\frac{R_k^2}{4t}}.
\eeq

We now prove the convergence of the weighted sum defining $\mathcal{R}_t$ in \eqref{eq:M_t,R_t}. To this aim, recalling
the asymptotic relation proved in Proposition~\ref{prop:ub}, we can say that for $t$ sufficiently large:
\beq\label{eq:la4W(QRk+1)bound}
\exp \lc \dfrac{(t-1)}{4} \la_{4\dot{W}}^{+} \lp Q_{R_{k+1}} \rp \rc \leq \exp \lc \dfrac{(t-1)}{4} \lp 16 (2 c_H \ce)^{\frac{1}{1+H}} + 1\rp {\lp \log (R_{k+1}) \rp}^{\frac{1}{1+H}}\rc.
\eeq
Consequently, using our bound \eqref{eq:al_k-bnd} on $\al_k$ and the expression \eqref{eq:R_k} for $R_k$ we have the following:
\beq\label{eq:R-t-ineq}
\mathcal{R}_t \leq \sum_{k=1}^{\infty} A_{k,t} B_{k,t} C_{k,t},
\eeq
where 
\begin{align}\label{eq:A_k,t}
A_{k,t} &= c R_{k+1}^{\frac{1}{4}} = c M^{\frac{k+1}{4}} t^{\frac{k+1}{4}} {\lp \log t \rp}^{\frac{k+1}{8(1+H)}} \leq c M^{\frac{k}{2}} t^{\frac{k}{2}} {\lp \log t \rp}^{\frac{k}{4(1+H)}},\\
B_{k,t} &= \exp {\lp -\frac{R_k^2}{4t} \rp} = \exp \lc -\frac{1}{4} M^{2k} t^{2k-1} (\log t)^{\frac{k}{1+H}} \rc, \nonumber \\
C_{k,t} &= \exp \lc \dfrac{(t-1)}{4} \la_{4\dot{W}}^{+} \lp Q_{R_{k+1}} \rp \rc. \nonumber
\end{align}
Furthermore, thanks to \eqref{eq:la4W(QRk+1)bound} for $t$ large enough we have:
$$
C_{k,t} \leq \exp \lc c\, t {\lp\log R_{k+1}\rp}^{\frac{1}{1+H}} \rc.
$$
Thus, plugging the value \eqref{eq:R_k} of $R_k$ into the above inequality we get
$$
C_{k,t} \leq \exp \lc  c\, t (k+1)^{\frac{1}{1+H}} {\lp \log M + \log t + \dfrac{1}{2(1+H)} \log\log t \rp}^{\frac{1}{1+H}} \rc.
$$
It is then readily checked for large enough $t$, that  
$$
C_{k,t} \leq \exp \lc c (k+1)^{\frac{1}{1+H}}  t (\log t)^{\frac{1}{1+H}}\rc.
$$
In addition, it is easily seen that for any arbitrary constant $c>0$, there exists $M$ large enough such that $M^{2k} > 8c{(k+1)}^{\frac{1}{1+H}}$ uniformly in $k$. Therefore for this value of $M$ and for all $k \geq 1$ we have
\begin{align}\label{eq:B_k,tC_k,t}
B_{k,t} C_{k,t} &\leq \exp \lc c (k+1)^{\frac{1}{1+H}}  t (\log t)^{\frac{1}{1+H}}  -\frac{1}{4} M^{2k} t^{2k-1} (\log t)^{\frac{2k}{2(1+H)}}\rc \nonumber\\
&\leq \exp \lc -\dfrac{1}{8} M^{2k} t^{2k-1} (\log t)^{\frac{2k}{2(1+H)}} \rc \leq \exp \lc -\dfrac{1}{8} M^{k} t^{k} (\log t)^{\frac{k}{2(1+H)}} \rc.
\end{align}
Combining \eqref{eq:A_k,t} and \eqref{eq:B_k,tC_k,t} we have thus obtained
$$
A_{k,t} B_{k,t} C_{k,t} \leq c_1 \eta_t^k e^{-c_2 \eta_t^{2k}},~\text{where}~\eta_t = \sqrt{M t (\log t)^{\frac{1}{2(1+H)}}}.
$$
Furthermore, we have that $\eta_t^{2k} > k \eta_t$ for all positive integers $k$ if $\eta_t > \sqrt{2}$. Thus, for sufficiently large $t$:
$$
A_{k,t} B_{k,t} C_{k,t} \leq c_1 \eta_t^k e^{-c_2 {k \eta_t}}.
$$
Recalling \eqref{eq:R-t-ineq}, the following bound holds true for the term $\mathcal{R}_t$ defined by \eqref{eq:UB-4}-\eqref{eq:M_t,R_t}:
\beq\label{eq:R-t-bnd}
\mathcal{R}_t \leq  c_1 \sum_{k=1}^{\infty} {\lp \eta_t e^{-c_2 \eta_t} \rp}^k < 2
\eeq
for all sufficiently large $t$ such that $\eta_t e^{-c_2 \eta_t}< \frac{1}{2}$.

Now let us work with the term $\mathcal{M}_{p,t}$ in the expression \eqref{eq:UB-4}. Observe that using Theorem~\ref{thm:Introthm-2} we have:
\begin{align}\label{eq:u_t-decomp-part1}
\lim_{t \to \infty} \dfrac{\la_{p \dot{W}} \lp Q_{R_1} \rp}{\la_{\dot{W}} (Q_t)} &= \lim_{t \to \infty} \dfrac{\la_{p \dot{W}} \lp Q_{R_1} \rp}{{\lp \log(R_1)\rp}^{\frac{1}{1+H}}} \dfrac{{\lp \log(t) \rp}^{\frac{1}{1+H}}}{\la_{\dot{W}}(Q_t)} \dfrac{{\lp \log(R_1)\rp}^{\frac{1}{1+H}}}{{\lp \log(t) \rp}^{\frac{1}{1+H}}}\nonumber \\
&= \dfrac{(2 c_H p^2 \ce)^{1/(1+H)}}{(2 c_H \ce)^{1/(1+H)}}\nonumber \\
&= p^{\frac{2}{1+H}},
\end{align}
where we have also used the form of $R_k$ from \eqref{eq:R_k} to show that the limit of $\log (R_1)/ \log(t)$  as $t$ goes to infinity is $1$. Plugging this identity into the definition \eqref{eq:M_t,R_t}, we get that 
\beq\label{eq:M_p,t-asymp_behav}
\dfrac{1}{t} \log (\mathcal{M}_{p,t}) \sim p^{\frac{1-H}{1+H}} \la_{\dot{W}} (Q_t) ,
\eeq
as $t$ goes to infinity. In particular, owing to Theorem~\ref{thm:Introthm-2} we have
$$
\lim_{t \to \infty} \dfrac{1}{t} \log(\mathcal{M}_{p,t}) = \infty \hspace{0.2in} \textnormal{a.s.}
$$
\noindent
Finally, going back to \eqref{eq:UB-4} we write
\beq\label{eq:log-u_t-asymp-1}
\dfrac{1}{t} \log \lp u_t(0) \rp \leq \dfrac{1}{t} \log \lp \mathcal{M}_{p,t} \rp + \dfrac{1}{t} \log \lp a_{1,p,q} + a_2 \dfrac{\mathcal{R}_t}{\mathcal{M}_{p,t}} \rp.
\eeq
Due to \eqref{eq:R-t-bnd} and the fact that $\lim_{t \to \infty} \cm_{p,t} = \infty$, we have
$$
\lim_{t \to \infty} \dfrac{1}{t} \log \lp a_{1,p,q} + a_2 \dfrac{\mathcal{R}_t}{\cm_{p,t}} \rp = 0.
$$
Therefore, thanks to \eqref{eq:M_p,t-asymp_behav}, relation \eqref{eq:log-u_t-asymp-1} entails the following upper bound:
\beq\label{eq:p-dep-UB}
\limsup_{t \to \infty} \dfrac{\frac{1}{t} \log (u_t(0))}{\la_{\dot{W}} (Q_t)} \leq p^{\frac{1-H}{1+H}}.
\eeq
At the very end, notice that the parameter $p > 1$ in \eqref{eq:p-dep-UB} can be chosen arbitrarily close to $1$. Therefore, taking limits as $p \downarrow 1$ in \eqref{eq:p-dep-UB}, we end up with
$$
\limsup_{t \to \infty}  \dfrac{\frac{1}{t}\log(u_t(0))}{\la_{\dot{W}} (Q_t)} \leq 1. \hspace{0.2in}\textnormal{a.s.},
$$
which is our claim \eqref{eq:UB-result}.
\end{proof}

\subsection{Lower Bound.}
This section is devoted to finding a lower bound for $\log \lp u_t(0) \rp$ matching the upper bound \eqref{eq:UB-result}. Specifically we will get the following result.
\begin{proposition}
Let $u_t$ be the field defined by \eqref{eq:f-k}. Then the following holds:
\beq\label{b1}
\liminf_{t \to \infty}\dfrac{\frac{1}{t}\log(u_t(0))}{\la_{\dot{W}} (Q_t)} \geq 1,
\eeq
where we recall that $Q_t = (-t, t)$ and $\la_{\dot{W}}(D)$ is introduced in \eqref{eq:la}.
\end{proposition}
\begin{proof}
\noindent
Let $p,q >1$ satisfy $\frac{1}{p} + \frac{1}{q} = 1$ with $p$ close to $1$ and let $0 < b <1$ be close to $1$. From Lemma~\ref{lem:1}(ii), taking $\al = p$, $q = \beta$, $\der = t^b$ and $\xi = \dot{W}^{\ep}$, we get
\begin{multline}
u_t^{\ep}(0) = \E_0 \lc \exp{\lp \int_0^t \dot{W}^{\ep} (B_s)ds \rp} \rc \geq {\lp \E_0 \lc \exp \lp -\dfrac{q}{p} \int_0^{t^b} \dot{W}^{\ep} (B_s) ds \rp \rc \rp}^{-\frac{p}{q}} \\
\times {\lcl \int_{Q_{t^b}} p_{t^b}(x) \E_x \lc \exp \lp \dfrac{1}{p} \int_0^{t-t^b}\dot{W}^{\ep} (B_s) ds \rp \mathbf{1}_{\lcl \tau_{Q_{t^b}} \geq t-t^b \rcl} \rc dx\rcl}^p,
\end{multline}
where we recall that $p_{\der}$ is the heat kernel in $\R$. Hence some elementary bounds on $p_{\der}$ over $Q_{t^b}$ yield
\beq\label{eq:LB-1}
u_t^{\ep}(0) \geq D_{\ep, b, p, t} F_{\ep, b, p, t},
\eeq
where 
\begin{align*}
D_{\ep, b, p, t} &= {\lp \E_0 \lc \exp \lp -\dfrac{q}{p} \int_0^{t^b} \dot{W}^{\ep} (B_s) ds \rp \rc \rp}^{-\frac{p}{q}},\\
F_{\ep, b, p, t} &= {\lcl \dfrac{e^{-t^b/2}}{\sqrt{2 \pi t^b}} \int_{Q_{t^b}} \E_x \lc \exp \lp \dfrac{1}{p} \int_0^{t-t^b}\dot{W}^{\ep} (B_s) ds \rp \rc dx\rcl}^p.
\end{align*}
We will now bound $D_{\ep,b,p,t}$ and $F_{\ep,b,p,t}$ separately.
\noindent
In order to bound $F_{\ep, b, p, t}$ we apply Lemma~\ref{lem:0}(ii), taking $\al = p$, $\beta = q$, $t = t-t^b$ and $\der = t^b$:
\begin{multline*}
\int_{Q_{t^b}} \E_x \lc \exp {\lp \int_0^{t-t^b} \dot{W}^{\ep} (B_s) ds \rp} \rc dx \\
\geq (2 \pi)^{\frac{p}{2}} t^{\frac{b}{2}} (t-t^b)^{\frac{p}{2q}} (2 t^b)^{-\frac{2p}{q}}
 \exp \lp -t^b \dfrac{p}{q} \la_{\frac{p}{q} \dot{W}^{\ep}} \lp Q_{t^b} \rp \rp \exp \lp p t \la_{\frac{\dot{W}^{\ep}}{p}} \lp Q_{t^b} \rp \rp.
\end{multline*} 

\noindent
Using \eqref{eq:LB-1} and replacing $e^{- \frac{t^b}{2}}$ by $e^{-C t^b}$ for a larger $C$ to absorb all bounded-by-polynomial quantities, we thus get
\beq\label{eq:F-bnd}
F_{\ep, b, p, t} \geq e^{-Ct^b} \exp\lp -\dfrac{p^2t^b}{q} \la_{\frac{p}{q} \dot{W}^{\ep}} \lp Q_{t^b} \rp\rp \exp \lp t \la_{\frac{\dot{W}^{\ep}}{p}} \lp Q_{t^b} \rp \rp.
\eeq
We now take limits as $\ep \downarrow 0$ in relation \eqref{eq:LB-1}. Invoking Proposition~\ref{prop:exp_moment_conv}, we use our bound \eqref{eq:F-bnd} and Lemma~\ref{lem:-1} which gives 
\beq\label{eq:LB-after_lim_ep}
u_t(0) \geq D_{b,p,t} F_{b,p,t}
\eeq
with
\begin{align*}
D_{b,p,t} &= e^{-Ct^b} {\lp \E_0 \lc \exp \lcl -\dfrac{q}{p} \int_0^{t^b} W(\der_{B_s}) ds \rcl \rc \rp}^{-\frac{p}{q}} ,\\
F_{b,p,t} &= \exp \lp -\dfrac{p^2 t^b}{q} \la_{\frac{p}{q} \dot{W}}^{+} \lp Q_{t^b} \rp \rp \exp \lp t \la_{\frac{\dot{W}}{p}} \lp Q_{t^b} \rp \rp.
\end{align*}

We will now prove that
\beq\label{eq:lim_Dbpt}
\lim_{t \to \infty} \dfrac{1}{t} \log(D_{b,p,t}) = 0
\eeq
Indeed, it is easily seen that 
\beq\label{eq:1-t_log(Dbpt)}
\dfrac{1}{t}\log(D_{b,p,t}) = -\dfrac{C}{t^{1-b}} - \dfrac{p}{q t^{1-b}} \dfrac{\log \lp \E_0 \lc \exp\lcl  -\dfrac{q}{p} \int_0^{t^b} W(\der_{B_s})ds \rcl \rc \rp}{t^b}.
\eeq
Moreover combining \eqref{eq:UB-result} and Proposition~\ref{prop:ub} we get the following bound for $t$ large enough:
$$
\dfrac{\log \lp \E_0 \lc \exp\lcl  -\dfrac{q}{p} \int_0^{t^b} W(\der_{B_s})ds \rcl \rc \rp}{t^b} \leq c (\log t)^{\frac{1}{1+H}}
$$
Plugging this information into \eqref{eq:1-t_log(Dbpt)}, we obtain \eqref{eq:lim_Dbpt}.

Let us now analyse the term $F_{b,p,t}$ in \eqref{eq:LB-after_lim_ep}. We have
$$
\dfrac{1}{t} \log (F_{b,p,t}) = - \dfrac{p^2}{q} \dfrac{ \la_{\frac{p}{q} \dot{W}}^{+} \lp Q_{t^b} \rp}{t^{1-b}} + \la_{\frac{\dot{W}}{p}} \lp Q_{t^b} \rp.
$$
Taking into account the behavior of $\la_{\frac{\dot{W}}{p}}(Q_t)$ given by Proposition~\ref{prop:ub}, we get
\beq\label{eq:lim_Fbpt}
\liminf_{t \to \infty} \dfrac{\frac{1}{t}\log (F_{b,p,t})}{\la_{\frac{\dot{W}}{p}}(Q_{t^b})} \geq 1\hspace{0.2in}\textnormal{a.s.}
\eeq

In conclusion, plugging \eqref{eq:lim_Fbpt} and \eqref{eq:lim_Dbpt} into \eqref{eq:LB-after_lim_ep}, we end up with
\beq\label{eq:before_lim_p,b}
\liminf_{t \to \infty} \dfrac{\frac{1}{t}\log (u_t(0))}{\la_{\frac{\dot{W}}{p}} (Q_{t^b})} \geq 1 \hspace{0.2in} \textnormal{a.s.}
\eeq
Now taking $b \uparrow 1$ and then $p \downarrow 1$ in \eqref{eq:before_lim_p,b}, and observing that $\la$ is monotonic under both maneuver, we get our desired lower bound \eqref{b1}:
$$
\liminf_{t \to \infty} \dfrac{\frac{1}{t}\log(u_t(0))}{\la_{\dot{W}} \lp Q_{t} \rp} \geq 1 \hspace{0.2in}\textnormal{a.s.}
$$
\end{proof}

\end{document}